\documentclass[11pt,a4paper]{article}
\usepackage{mathrsfs}
\usepackage{mathrsfs}
\usepackage{mathrsfs}
\usepackage{amsfonts}
\usepackage{amsfonts}
\usepackage{amsfonts}
\usepackage{mathrsfs}
\usepackage{amssymb}
\usepackage{color}
\usepackage{amsmath}
\usepackage{amssymb,amsmath,amsthm, amsfonts}
\usepackage{subfigure}
\usepackage{graphicx}
\usepackage{subfigure}
\usepackage{color}
\usepackage{multicol}
 \usepackage{mathrsfs}
 \allowdisplaybreaks[4]
 \usepackage{dsfont}

\newtheorem{theorem}{Theorem}[section]

\numberwithin{equation}{section}

\newtheorem{lemma}[theorem]{Lemma}

\newtheorem{proposition}[theorem]{Proposition}
\newtheorem{remark}[theorem]{Remark}

\begin{document}
\parindent 9mm
\title{Stabilization for a Flexible Beam with Tip Mass and Control Matched Disturbance
\thanks{This work was supported by the Natural Science Foundation of Shaanxi Province (grant nos.
2018JM1051, 2014JQ1017), the Fundamental Research Funds for the Central Universities (Grant no.
xjj2017177)} }
\author{Zhan-Dong Mei \thanks{School of Mathematics and Statistics, Xi'an Jiaotong
University, Xi'an 710049, China; Email: zhdmei@mail.xjtu.edu.cn }}


\date{}
\maketitle \thispagestyle{empty}
\begin{abstract}
%
%

This paper is concerned with the output feedback exponential stabilization for a flexible beam with tip mass.
When there is no disturbance, it is shown that only one non-collocated measurement is enough to exponentially stabilize
the original system by constructing an infinite-dimensional Luenberger state observer to track the state and designing an estimated state based output feedback control law. This essentially improves the existence result in [F. Conrad and O. M\"{o}rg\"{u}l, SIAM
J. Control Optim., 36 (1998), 1962-1986] where two collocated measurements including high order feedback were adopted. In the case that boundary internal uncertainty and external disturbance are considered, an infinite-dimensional disturbance estimator is constructed to estimate the state and total disturbance in real time. By virtue of the estimated state and estimated total disturbance, an output feedback control law is designed to exponentially stabilize the original system while guaranteeing the boundedness of the closed-loop system. Some illustration simulations are presented.

\vspace{0.5cm} 

%
%

\end{abstract}


\section{Introduction}\label{section1}

This paper is concerned with the dynamic stabilization for a flexible beam with
tip mass, which describes the SCOLE model in the sense that the moment of inertia at $x=1$ is neglected.
Such system with boundary control matched internal uncertainty and external disturbance can be written mathematically as follows
\begin{equation} \label{beem}
\left\{\begin{array}{l}
w_{tt}(x,t)+w_{xxxx}(x,t)=0,\;\; x\in (0,1), \; t>0, \\
w(0,t)=w_x(0,t)=w_{xx}(1,t)=0, \; \;   t\ge 0,\\
-w_{xxx}(1,t)+mw_{tt}(1,t)=u(t)+F(t), \;\;  t\ge 0,\\
{\color{blue}y(t)=(w_{xx}(0,t),w(1,t))^T}, \;\;  t\ge 0,
\end{array}\right.
\end{equation}
where $w(x,t)$ is the displacement of the beam at position $x$ and time $t$, $m>0$ is the tip mass,
$u(t)$ is the boundary shear control (or input) at the free end of the beam, $F(t)=v(w(\cdot,t),w_t(\cdot,t))+d(t)$, $v:H^2(0,1)\times L^2(0,1)\rightarrow \mathds{R}$ is the internal uncertainty, $d(t)$ is the external disturbance, $y(t)$ is the observation (output) with $w_{xx}(0,t)$ being bending strain at $x=0$. {\color{blue}System (\ref{beem}) can be derived by Hamilton principle [11,12].}

We consider system (\ref{beem}) in the state Hilbert space $\mathbf{H}_1=H_E^2(0,1)\times L^2(0,1)\times\mathds{R}$ with inner product induced norm given by
$\|(f,g,\eta)\|^2_{\mathbf{H}_1} = \int_0^1[|f''(x)|^2$
$+|g(x)|^2]dx+\frac{1}{m} |\eta|^2,(f,g,\eta)\in \mathbf{H}_1,$ where $\ H_E^2(0,1)=$
$\{f\in H^2(0,1):f(0)=f'(0)=0\}$.
Define the operator $\mathbf{A}_1:D(\mathbf{A}_1)(\subset \mathbf{H}_1)\rightarrow \mathbf{H}_1$ by
$\mathbf{A}_1(f,g,\eta)=(g,-f^{(4)},f'''(1)),\; \forall\;{\color{blue}(f,g,\eta)}\in D(\mathbf{A}_1),$
   $D(\mathbf{A}_1)=\{(f,g,\eta)\in {\color{blue}(H^4(0,1)\bigcap H^2_E(0,1))\times H^2_E(0,1)\times \mathds{R}}|f''(1)=0,\eta=mg(1)\}.$
It is routine to check that $\mathbf{A}_1$ is a skew-adjoint operator and thereby it generates a unity group.
System (\ref{beem}) is then abstractly written as
\begin{align*}
   & \frac{d}{dt}\big(w(\cdot,t),
                w_t(\cdot,t),
                mw_t(1,t)\big)
=\mathbf{A}_1\big(w(\cdot,t),
                w_t(\cdot,t),\\
                & mw_t(1,t)\big)+\mathbf{B}_1[u(t)+F(t)],
\end{align*}
where $\mathbf{B}_1=(0,0,1)^T$ is a bounded linear operator.
\begin{proposition}
Suppose that $v:H^2(0,1)\times L^2(0,1)\rightarrow \mathds{R}$ is continuous and satisfies global Lipschitz condition in $H^2(0,1)\times L^2(0,1)$. Then, for any $(w(\cdot,0),w_t(\cdot,0),$
$mw_t(1,0))\in \mathbf{H}_1$, $u,d \in L^2_{loc}(0,\infty)$,
there exists a unique global solution (mild solution) to (\ref{beem}) such that $(w(\cdot,t),w_t(\cdot,t),$
$mw_t(1,t))\in C(0,\infty;\mathbf{H}_1)$.
\end{proposition}
\begin{proof}\ \
The proof can be obtained directly from the procedure of \cite[Proposition 1.1]{Zhou2018a}.
\end{proof}

In the case that the disturbance is not considered ($F(t)\equiv 0$), system (\ref{beem}) has been extensively discussed \cite{Balakrishnan1991,Chentouf2006,Conrad1998,Guo2001b,Littman1988,Rao1995,Triggiani1989,Zhao2010}.
It is well-know that if $m=0$ (no tip mass is considered), then (\ref{beem}) devolves into Euler-Bernoulli beam equation with shear boundary control, and the boundary velocity feedback $u(t)=-\alpha w_t(1,t)(\alpha>0)$ has been proved to exponentially stabilize the original system by virtue of multiplier method or Riesz basis approach \cite{Chen1987,Guo2001}. However, in presence of tip mass, boundary velocity feedback $u(t)=-\alpha w_t(1,t) $ can be regarded as a compact perturbation of the free system ($u(t)=0$) \cite{Rao1995}; such a compact perturbation makes the closed-loop system strongly stable \cite{Littman1988} but {\color{blue}cannot} guarantee the exponential stability \cite{Triggiani1989}. In order to exponentially stabilize the system, apart from the velocity, high order feedback such as $w_{xxxt}(1,t)$ should also be considered \cite{Rao1995}. In \cite{Conrad1998}, it was proved by virtue of energy multiplier approach that for any $\beta>0$ the closed-loop system under the feedback control law $u(t)=-\alpha w_t(1,t)+\beta w_{xxxt}(1,t)$, is exponentially stable. Furthermore, the authors in \cite{Conrad1998} also proved that there exist a set of generalized eigenfunctions of the closed-loop system that forms a Riesz basis provided $m=\alpha\beta$. By virtue of their Riesz basis generation theory in \cite{Guo2001}, Guo \cite{Guo2001b} proved that the Riesz basis property indeed hold for all the cases $m,\alpha,\beta>0$.

The aforemention literatures considered collocated control and observation, where the actuators and sensors lie in the same boundary place $x=1$.  The collocated design approach was firstly introduced in circuit theory in 1950s \cite{Guillemin1957}.
However, in practical control, the performance of closed-loop system under collocated output feedback may be not so good \cite{Cannon1984}; non-collocated control design has been widely used \cite{Cannon1984,Wu2001}.
Since the closed-loop system under non-collocated design is usually non-dissipative, it is hard to apply the traditional Lyapunov method or multiplier approach to discuss the stability. To overcome this difficulty, instead of direct output feedback control design, compensator based controller design method is a possible choice. In 1975, Gressang and Lamont \cite{Gressang1975} firstly presented in semigroup frame a generalized Luenberger stabilizing compensator for infinite-dimensional systems with bounded input and output operators.
The authors in \cite{Curtain1986,Lasiecka1995} discussed finite-dimensional compensators for infinite-dimensional linear system with unbounded input and bounded/unbounded output operators.  By virtue of backstepping observers, Smyshlyaev and Krstic \cite{Smyshlyaev2005}
constructed an estimated state feedback controller to stabilize a class of one-dimensional parabolic PDEs.
In \cite{Deguenon2006}, an abstract observer was designed for a class of well-posed regular infinite-dimensional systems. Motivated by \cite{Deguenon2006} and \cite{Smyshlyaev2005}, Guo, Wang and Yang proved in \cite{Guo2008} that the measurement $y(t)=w_{xx}(0,t)$ can be used to exponentially stabilize the original system with $m=0$ by designing an infinite dimensional state observer and using estimated state based feedback control law.
However, because of the complexity, for the case $m\neq0$, the non-collocated control of (\ref{beem}) is still a long standing unsolved problem.
In the appendix, with only one non-collocated measurement $w_{xx}(0,t)$, we shall present an estimated state based output feedback control law (\ref{control}) to exponentially stabilize system (\ref{beem}) with $F(t)\equiv 0$. This is a great improvement of the references \cite{Conrad1998,Guo2001b}, where two collocated measurements $w_t(1,t)$ and $w_{xxxt}(1,t)$ with $w_{xxxt}(1,t)$ being high order feedback were used.

Because of the influences of both model uncertainty and external environment, internal uncertainty and external disturbance should also be taken into consideration for practical control problem ($F(t)\neq 0$). However, even a small amount of disturbance in the stabilizing boundary output feedback design could destabilize the system. The stabilization of infinite-dimensional system with disturbance represents a difficult mathematical challenge. There emerged many well-developed methods coping with uncertainty in control system problem, such as Lyapunov approach \cite{Guo2014a,Jin2015}, sliding mode control (SMC) \cite{Cheng2011,Guo2013a}, backstepping approach \cite{GuoW2011a,GuoW2013b}, adaptive control \cite{Bresch-Pietri2014,GuoW2013a,Krstic2010} and active disturbance rejection control (ADRC) \cite{Feng2017a,Guo2013a,Han2009,Zhou2018a,Zhou2020}. Since the core idea is real time estimating/cancelation of uncertainties through an extended state observer (ESO) and compensating the disturbance in feedback loop, ADRC can significantly reduce the control energy. In the earlier efforts for PDEs by ADRC like \cite{Guo2013b,Jin2015}, the disturbance was dealt with by ODEs reduced from the associated PDEs through some special test functions for which variation of the external disturbance is supposed to be slow and more critically, the high gain must be used in ESO to estimate the total disturbance; the internal uncertainties {\color{blue}cannot} be coped with by such method.
In \cite{Feng2017a}, a new infinite-dimensional disturbance estimator was developed to relax such restricts of the conventional ESO for a class of anti-stable wave equations with external disturbance and three measurements. Later on, Zhou and Weiss \cite{Zhou2018b} used the same method to improve the result in \cite{Feng2017a} with only two measurements. The approach was then  applied to multi-dimensional
wave equation with internal uncertainty and external disturbance \cite{Zhou2017a}, Euler-Bernoulli beam equation with boundary shear force and internal uncertainty and external disturbance being considered simultaneously \cite{Zhou2018a}, and Euler-Bernoulli beam equation with
moment boundary control matched internal uncertainty and external disturbance \cite{Zhou2020}.

In the appendix, it points out that, in presence of disturbance, the stabilizing control law (\ref{control}) should be redesigned. System (\ref{beem}) with $v(w,w_t)=0$ and $d(t)$ being uniformly bounded was firstly studied by Ge et al. \cite{Ge2011a,Ge2011b}; the adaptive boundary control strategy was applied to cope with system parametric uncertainties and disturbance observer for attenuating the effect of the external disturbances. In \cite{Li2017}, Li, Xu and Han
discussed system (\ref{beem}) with $v(w,w_t)=0$ and $m=\alpha\beta$, but the internal uncertainty as well as the general case $m\neq \alpha\beta$ has not been taken into consideration. Moreover, \cite{Li2017} used three measurements including the high order collocated feedbacks $w_{xxx}(1,t)$ and $w_{xxxt}(1,t)$. In this paper, by using two measurements $w_{xx}(0,t)$ and $w(1,t)$, an infinite-dimensional disturbance estimator shall be constructed to estimate the original state and total disturbance in real time. Based on the estimated state and estimated total disturbance, an output feedback control law is designed to exponentially stabilize the original system while guarantee the boundedness of the closed-loop system. {\it The main contribution for the system (\ref{beem}) with disturbance ($F(t)\neq 0$) is: 1) the internal uncertainty is taken into consideration, 2) the high order collocated feedbacks $w_{xxx}(1,t)$ and $w_{xxxt}(1,t)$ are not used, 3) we consider arbitrary $m,\alpha,\beta>0$ while \cite{Li2017} just solved the special case $\alpha,\beta>0$ and $m=\alpha\beta$,
4) since no additional state observer is used, our control strategy is concise and energy-saving, and our method provides an idea to improve the results of \cite{Zhou2018a,Zhou2018b}, see Remark \ref{xiangdui}.}

The arrangement of this paper is as follows. In Section 2, we present an infinite-dimensional Luenberger state observer for system (\ref{beem}) to estimated total disturbance and state in real time. Moreover an estimated total disturbance and estimated state based control law is designed.
In Section 3, the exponential stability of a couple subsystem including the original equation of the closed-loop is concluded. Moreover, the other state of the closed-loop system is proved to be bounded. In Section 4, we present some simulations in order to illustrate our theory results.
In the appendix, we give the observer and controller design for system (\ref{beem}) without disturbance.

\section{Disturbance estimator and controller design}\label{disturbancestab}
In this section we shall design an infinite-dimensional disturbance estimator to estimate the total disturbance $F(t)$ of system (\ref{beem}).
We first introduce an auxiliary system to transfer the total disturbance $F(t)$ into an exponentially stable system:
\begin{equation} \label{transferl}
\left\{\begin{array}{l}
l_{tt}(x,t)+l_{xxxx}(x,t)=0, \\
l(0,t)=l_{xx}(1,t)=0, \\
l_{xx}(0,t)=c l_{xt}(0,t)+\gamma l_x(0,t)+w_{xx}(0,t),  \\
-l_{xxx}(1,t)+ml_{tt}(1,t)=u(t),
\end{array}\right.
\end{equation}
where $c,\; \gamma>0$ are tuning parameters. Obviously, the auxiliary system (\ref{transferl}) is the same as the observer (\ref{transfer11}) designed for the system without disturbance.

Set $\widehat{l}(x,t)=l(x,t)-w(x,t)$ to get
\begin{equation} \label{perror}
\left\{\begin{array}{l}
\widehat{l}_{tt}(x,t)+\widehat{l}_{xxxx}(x,t)=0, \\
\widehat{l}(0,t)=\widehat{l}_{xx}(1,t)=0,\\
\widehat{l}_{xx}(0,t)=c \widehat{l}_{xt}(0,t)+\gamma \widehat{l}_x(0,t),\\
-\widehat{l}_{xxx}(1,t)+m\widehat{l}_{tt}(1,t)=-F(t).
\end{array}\right.
\end{equation}
We consider system (\ref{transfer11}) in the state Hilbert space $\mathbf{H}_2=H_L^2(0,1)\times L^2(0,1)\times\mathds{R},\ H_L^2(0,1)=\{f|f\in H^2(0,1),f(0)=0\}$, with inner product induced norm given by
$\|(f,g,\eta)\|^2_{\mathbf{H}_2} = \int_0^1[|f''(x)|^2$
$+|g(x)|^2]dx+\gamma |f'(0)|^2+\frac{|\eta|^2}{m},$
$(f,g,\eta)\in \mathbf{H}_2.$

Then (\ref{perror}) is written abstractly by
\begin{align}\label{perrorsemigroup}
    \nonumber&\frac{d}{dt}(\widehat{l}(\cdot,t),\widehat{l}_t(\cdot,t),m\widehat{l}(1,t))\\
    &=\mathbf{A}_2(\widehat{l}(\cdot,t),\widehat{l}_t(\cdot,t),m\widehat{l}(1,t))-\mathbf{B}_{2}F(t),
\end{align}
where $\mathbf{A}_2:D(\mathbf{A}_2)(\subset \mathbf{H}_2)\rightarrow \mathbf{H}_2$ is defined by
$\mathbf{A}_2(f,g,\eta)=(g,-f^{(4)},f'''(1)),\; \forall\;(f,g,\eta)\in D(\mathbf{A}_2),$
   ${\color{blue}D(\mathbf{A}_2)=\{(f,g,\eta)\in (H^4(0,1)\times H^2_L(0,1))\times H^2_L(0,1)\times \mathds{R}|}$
   $ f''(0)=c g'(0)+\gamma f'(0),
   f''(1)$
   $=0,\eta=mg(1))\},$
$\mathbf{B}_2=(0,0,1)^T$ is a bounded linear operator.
Use \cite[Lemma A.1 and A.2]{Zhou2018a}, we derive the following lemma.

\begin{lemma}\label{admissible}
Assume that $d\in L^\infty(0,\infty)$ (or $d\in L^2(0,\infty))$, $f: H^2(0,1)\times L^2(0,1)\rightarrow \mathds{R}$ is continuous and system (\ref{beem}) admits a unique bounded solution $(w(\cdot,t),w_t(\cdot,t))\in C(0,\infty;H^2(0,1)\times L^2(0,1))$. Then for any initial value $(\widehat{l}(\cdot,0),\widehat{l}_t(\cdot,0),m\widehat{l}_t(1,0))\in \mathbf{H}_2$, there exists a unique solution
to system (\ref{perror}) such that $(\widehat{l}(\cdot,t),\widehat{l}_t(\cdot,t),m\widehat{l}_t(1,t))\in C(0,\infty;\mathbf{H}_2)$ and
$\|(\widehat{l}(\cdot,t),\widehat{l}_t(\cdot,t),m\widehat{l}_t(1,t))\|_{\mathbf{H}_2}$
$< +\infty$.
Moreover, if $\lim_{t\rightarrow \infty}v(w(\cdot,t),w_t(\cdot,t))=0$ and $d\in L^2(0,\infty)$,
then $\lim_{t\rightarrow \infty}\|(\widehat{l}(\cdot,t),\widehat{l}_t(\cdot,t),$
$m\widehat{l}_t(1,$
$t))\|_{\mathbf{H}_2}=0$.
\end{lemma}
\begin{proof}\ \
Since $\mathbf{B}_2$ is a bounded linear operator, the conclusions are obtained directly by the assumptions and \cite[Lemma A.1 and A.2]{Zhou2018a}.
\end{proof}

For error system (\ref{perror}) we design the following observer
\begin{equation} \label{errorobserver}
\left\{\begin{array}{l}
z_{tt}(x,t)+z_{xxxx}(x,t)=0,\\
z(0,t)=z_{xx}(1,t)=0,z(1,t)=l(1,t)-w(1,t),\\
z_{xx}(0,t)=c z_{xt}(0,t)+\gamma z_x(0,t).
\end{array}\right.
\end{equation}
Set $\widehat{z}(x,t)=z(x,t)-\widehat{l}(x,t)$ to derive
\begin{equation} \label{erroerrorobserver}
\left\{\begin{array}{l}
\widehat{z}_{tt}(x,t)+\widehat{z}_{xxxx}(x,t)=0, \\
\widehat{z}(0,t)=\widehat{z}(1,t)=\widehat{z}_{xx}(1,t)=0, \\
\widehat{z}_{xx}(0,t)=c \widehat{z}_{xt}(0,t)+\gamma \widehat{z}_x(0,t).
\end{array}\right.
\end{equation}
Consider system (\ref{erroerrorobserver}) in the space $\mathbb{H}={\color{blue}H^2_K(0,1)}\times L^2(0,1)$, ${\color{blue}H^2_K(0,1)}=\{f\in H^2(0,1)|f(0)=f(1)=0\}$ with inner product
induced norm $\|(f,g)\|^2_{\mathbb{H}}=\int_0^1[|f''(x)|^2+|g(x)|^2]dx+\gamma|f'(0)|^2$.
Define $\mathbb{A}(f,g)=(g,-f^{(4)}), \forall (f,g)\in D(\mathbb{A}),
   D(\mathbb{A})=\{(f,g)\in ({\color{blue}H^2_K(0,1)}\bigcap $
   $H^4(0,1))\times {\color{blue}H^2_K(0,1)}|f''(1)$
   $=0,f''(0)=c g'(0)+\gamma f'(0)\}.$
Then, system (\ref{erroerrorobserver}) can be written abstractly as
$\frac{d}{dt}(\widehat{z}(\cdot,t),\widehat{z}_t(\cdot,t))=\mathbb{A}(\widehat{z}(\cdot,t),\widehat{z}_t(\cdot,t)).$

\begin{lemma}\label{Riesz}
There exist a sequence of generalized eigenfunctions of $\mathbb{A}$ which forms Riesz basis for $\mathbb{H}$;
$\mathbb{A}$ is a generator of exponentially stable $C_0$-semigroup.
\end{lemma}
\begin{proof}\ \
It is routine to verify that $\mathbb{A}$ is a densely defined and discrete operator.
We split {\color{blue}the} rest of the proof into several steps and
omit some details.

{\bf Step 1.} We show that, for any $\lambda=i\tau^2\in\sigma(\mathbb{A})$, there corresponds one eigenfunction of the form
$(f,\lambda f)$, where $f$ is given by
 \begin{align}\label{fxequ3}
    \nonumber &f(x)=\sinh\tau(x-1)+\sin\tau \cosh\tau x-\cos\tau \sinh\tau x\\
    \nonumber&+\frac{2\tau\sin\tau}{ic\tau^2+\gamma}\sinh\tau x+\sin\tau(x-1)+\cos\tau x \sinh\tau\\
     &-\sin\tau x \cosh\tau-\frac{2\tau \sinh\tau}{ic\tau^2+\gamma}\sin\tau x,
  \end{align}
and $\tau$ satisfies the characteristic equation
  \begin{equation}\label{tauchar2}
  (ic\tau^2+\gamma)(\cosh\tau\sin\tau-\sinh\tau\cos\tau)+2\tau\sin\tau \sinh\tau=0,
  \end{equation}
which implies that each eigenvalue of $\mathbb{A}$ is geometrically simple.

{\color{blue}Indeed, $f$ satisfies the equation
\begin{equation}
  \begin{cases}
    f^{(4)}(x)-\tau^4f(x)=0,\\
    f(0)=f(1)=0,\\
    f^{''}(1)=0,\ f^{''}(0)=ic\tau^2f^{'}(0)+\gamma f^{'}(0).
  \end{cases}
\end{equation}
The solution of ode $f^{(4)}(x)-\tau^4f(x)=0$ is of the form
\begin{equation}\label{fxc1c2c3c400}
  f(x)=c_1e^{\tau x}+c_2e^{-\tau x}+c_3\cos\tau x+c_4\sin\tau x.
\end{equation}
From the boundary condition $f(0)=f(1)=0,f^{''}(0)=ic\tau^2f^{'}(0)+\gamma f^{'}(0)$ we can derive
\begin{equation}\label{c1c2c3c4}
  \begin{cases}
    c_1=\frac{1}{2}\left[e^{-\tau}-\cos\tau+\frac{2\tau\sin\tau}{ic\tau^2+\gamma}+\sin\tau\right],\\
    c_2=\frac{1}{2}\left[-e^{\tau}+\cos\tau-\frac{2\tau\sin\tau}{ic\tau^2+\gamma}+\sin\tau\right],\\
    c_3=\frac{1}{2}\left[e^\tau-e^{-\tau}-2\sin\tau\right],\\
    c_4=\frac{1}{2}\left[\frac{-2\tau}{ic\tau^2+\gamma}(e^\tau-e^{-\tau})-(e^\tau+e^{-\tau})+2\cos\tau\right].
  \end{cases}
\end{equation}
Substitute (\ref{c1c2c3c4}) into (\ref{fxc1c2c3c400}) to get (\ref{fxequ3}). The equation (\ref{tauchar2}) is derived by the boundary condition $f''(1)=0.$
}

{\bf Step 2.} We prove that the eigenvalues $\{\lambda_n,\overline{\lambda_n}\}$, $\lambda_n=i\tau_n^2$ of $\mathbb{A}$ have
asymptotic expression
\begin{equation}\label{ataunasy1}
    \tau_n=p\pi+O(n^{-1}),\; \lambda_n=-2/c+i(p\pi)^2+O(n^{-1}),
\end{equation}
where $n$ is a large positive integer and $p=n+1/4$;
the corresponding eigenfunction $(f_n,\lambda_nf_n)$ of $\mathbb{A}$ can be chosen by

 \begin{align}\label{fxequasy30}
    \begin{split}
 F_n(x)=\left(
              \begin{array}{c}
                -e^{-p\pi x}+ (\sin p\pi x-\cos p\pi x) \\
                i[-e^{-p\pi x}-(\sin p\pi x-\cos p\pi x)] \\
                0\\
              \end{array}
            \right)+O(n^{-1}), \tau_n^{-2}e^{-\tau_n}f'_n(0)=O(n^{-1}), \\
\end{split}
  \end{align}
 where
  \begin{align}\label{fmatrix1}
    F_n(x)=\frac{2}{\tau_n^2e^{\tau_n}}
    \left(
      \begin{array}{c}
        f_{n}^{''}(x) \\
      \lambda_nf_n(x)\\
      \gamma f_{n}^{'}(0)\\
      \end{array}
    \right),
    \; \lim_{n\rightarrow\infty}\|F_n\|^2_{L^2\times L^2\times \mathds{C}}=2.
  \end{align}

 {\color{blue} In fact,  When ${\rm Im}(\tau)$ is bounded and ${\rm Re}(\tau)\rightarrow\infty$, by (\ref{tauchar2}) we derive the following equality
\begin{align}\label{ruxie}
       \nonumber \cos\tau-\sin\tau=&\frac{(ic\tau^2+\gamma)e^{-\tau}+2\tau \sinh\tau}{(ic\tau^2+\gamma)\sinh\tau}\sin\tau=\frac{2\tau\sin\tau}{ic\tau^2+\gamma}+O(e^{-{\rm Re}\tau}) \\
          =&\frac{2\sin\tau}{ic\tau}+O(|\tau|^{-3}) =O(|\tau|^{-1}),
   \end{align}
which implies $sin(2\tau)=1+O(|\tau|^{-1})$.
Using Rouche's theorem to get $\tau_n=(n+\frac{1}{4})\pi+O(n^{-1})=p\pi+O(n^{-1})$, combining which and (\ref{ruxie}) to get $\tau_n=p\pi-\frac{1}{ic(p\pi)}+O(n^{-2})$. Hence (\ref{ataunasy1}) is proved.

By (\ref{ataunasy1}), for $y\geq0$ and $0\leq x\leq1$, we have
 \begin{equation}\label{ataunxmx}
   \begin{cases}
     e^{-\tau_ny}=e^{-p\pi y}+O(n^{-1}),\\
     \sin\tau_nx=\sin p\pi x+O(n^{-1}),\\
     \cos\tau_nx=\cos p\pi x+O(n^{-1}).
   \end{cases}
 \end{equation}
 Obviously, $f_n(x)=f(x)$ is defined by (\ref{fxequ3}) with $\tau=\tau_n$. By (\ref{ataunxmx}), it follows that
  \begin{align*}
    \frac{1}{\tau^2e^\tau}f^{''}(x)=&\frac{2}{e^\tau}[\sinh\tau(x-1)+\sin\tau \cosh\tau x-\cos\tau \sinh\tau x-\cos\tau x \sinh\tau+\sin\tau x \cosh\tau]\\
    &+\frac{2}{e^\tau}[\frac{2\tau\sin\tau}{ic\tau^2+\gamma}\sinh\tau x-\sin\tau(x-1)+\frac{2\tau sh\tau}{ic\tau^2+\gamma}\sin\tau x]\\
    =&[-e^{-\tau x}+\sin\tau e^{\tau(x-1)}-\cos\tau e^{\tau(x-1)}-\cos\tau x+\sin\tau x+O(e^{-{\rm Re}\tau})]+O(|\tau|^{-1})\\
    =&-e^{-p\pi x}+(\sin p\pi x-\cos p\pi x)+O(n^{-1}),
    \end{align*}
 and
 \begin{equation*}
  \begin{split}
       &\frac{1}{\tau^2e^\tau}(\lambda f(x))\\
       =&\frac{2i}{e^\tau}[\sinh\tau(x-1)+\sin\tau \cosh\tau x-\cos\tau \sinh\tau x+\cos\tau x \sinh\tau-\sin\tau x \cosh\tau] \\
       &+\frac{2i}{e^\tau}[\sin\tau(x-1)+\frac{2\tau\sin\tau}{ic\tau^2+\gamma}sh\tau x-\frac{2\tau \sinh\tau}{ic\tau^2+\gamma}\sin\tau x] \\
      =&i[-e^{-\tau x}+\sin\tau e^{\tau(x-1)}-\cos\tau e^{\tau(x-1)}+\cos\tau x-\sin\tau x+O(e^{-{\rm Re}\tau})]+O(|\tau|^{-1})\\
      =&i[-e^{-p\pi x}-(\sin p\pi x-\cos p\pi x)]+O(n^{-1}).
  \end{split}
\end{equation*}
Accordingly, (\ref{fxequasy30}) and (\ref{fmatrix1}) are obtained.
}

{\bf Step 3:} We show that for any $\lambda\in \sigma(\mathbb{A})$, ${\rm Re}\lambda<0$.
{\color{blue}To do this, we let $(f,g)$ be an eigenfunction corresponding to eigenvalue $\lambda$.
Then $g=\lambda f$ and
\begin{align}\label{1}
    f^{(4)}(x)+\lambda^2 f(x)=0,\ x\in[0,1].
\end{align}
Multiply (\ref{1}) by $\overline{f}$ and integral on both sides over $[0,1]$ to get
\begin{align}\label{10}
   \int_0^1f^{(4)}(x)\overline{f}(x)dx+({\rm Re}^2\lambda-{\rm Im}^2\lambda) \int_0^1 f(x)\overline{f}(x)dx+2i{\rm Re}\lambda {\rm Im}\lambda \int_0^1 f(x)\overline{f}(x)dx=0
\end{align}
with the boundary condition $f(0)=f(1)=f''(1)=0,\ f''(0)=[c({\rm Re}\lambda+i{\rm Im}\lambda)+\gamma]f'(0).$
Integral (\ref{10}) in part to derive
\begin{align}\label{l1}
   \nonumber &\int_0^1|f''(x)|^2dx+({\rm Re}^2\lambda-{\rm Im}^2\lambda) \int_0^1 |f(x)|^2dx+(c {\rm Re}\lambda\\
   &+\gamma) |f'(0)|^2
+i {\rm Im}\lambda\bigg[c|f'(0)|^2+2{\rm Re}\lambda \int_0^1 |f(x)|^2dx\bigg]=0.
\end{align}
Hence the imaginary part of (\ref{l1}) equals to zero, that is, ${\rm Im}\lambda=0$ or $c|f'(0)|^2+2{\rm Re}\lambda \int_0^1 |f(x)|^2dx=0$. We first show that
in both case ${\rm Re}\lambda\leq 0$.\\
Case 1: ${\rm Im}\lambda=0$. By the fact that the real part equals to zero, we derive $\int_0^1|f''(x)|^2dx+{\rm Re}^2\lambda \int_0^1 |f(x)|^2dx+(c{\rm Re}\lambda+\gamma)|f'(0)|^2=0$ which implies $(c {\rm Re}\lambda+\gamma)|f'(0)|^2\leq 0$. It is not hard to show that $f'(0)\neq 0$.  Indeed, if $f'(0)=0$, we can easily obtain  $f''(x)=0$; observe that $f(1)=0$. So $f=0$, which indicates that $(f,g)=0$ is
the eigenfunction. Then ${\rm Re}\lambda\leq 0$. \\
Case 2: $c|f'(0)|^2+2{\rm Re}\lambda \int_0^1 |f(x)|^2dx=0.$ The fact $f\neq 0$ indicates that ${\rm Re}\lambda\leq 0$.
}

{\color{blue}Next we shall prove that ${\rm Re}\lambda=0$ is impossible. In fact, ${\rm Re}\lambda=0$ implies
 \begin{align}\label{l1l}
   \int_0^1|f''(x)|^2dx-{\rm Im}^2\lambda \int_0^1 |f(x)|^2dx+\gamma|f'(0)|^2+ic {\rm Im}\lambda |f'(0)|^2=0.
\end{align}
Consider the imaginary part, ${\rm Im}\lambda=0$ or $f'(0)=0$.  If ${\rm Im}\lambda=0$, the real part implies
 $\int_0^1|f''(x)|^2dx+\gamma|f'(0)|^2=0$, combination with the condition $f(0)=f(1)=0$ to get $f(x)=0, x\in [0,1]$, thereby
$g(x)=0, \ x\in [0,1]$. If $f'(0)=0$, equation (\ref{1}) becomes
\begin{align*}
    \left\{
      \begin{array}{ll}
        f^{(4)}(x)-{\rm Im}^2\lambda f(x)=0, & \hbox{ } \\
        f(0)=f(1)=0, & \hbox{ } \\
        f'(0)=f''(0)=f''(1)=0, & \hbox{ }
      \end{array}
    \right.
\end{align*}
whose solution is $f(x)=0$. This indicates that $(f,\lambda f)$ is not an eigenfunction of $\lambda$ provided ${\rm Re}\lambda=0$.
}

{\bf Step 4:} We prove that the system operator
$J_1(f,g)$
$=(g,-f^{(4)}),
  D(J_1)=\{(f,g)\in(H^4(0,1)\cap H_K^2(0,1))\times$
  $ H_K^2(0,1)|f''(1)=g'(0)=0\}$
of reference system ($1/c=\gamma=0$)
\begin{equation}\label{fangonge}
  \begin{cases}
  \hat{v}_{tt}(x,t)+\hat{v}_{xxxx}(x,t)=0,\ 0<x<1,\\
  \hat{v}(0,t)=\hat{v}_{xt}(0,t)=0,\\
  \hat{v}_{xx}(1,t)=\hat{v}(1,t)=0
\end{cases}
\end{equation}
is skew-adjoint and have compact resolvent; all the eigenvalues  $\{\mu_n,\overline{\mu_n}\}$, $\mu_n=i\omega_n^2$ are algebraically simple and have asymptotic expression
\begin{equation}
  \omega_n=p\pi+O(e^{-n}), p=n+1/4;
\end{equation}
the eigenfunction have asymptotic expression
\begin{align}\label{aphixasy0}
         G_n(x)=\left(
                 \begin{array}{c}
                   -e^{-p\pi x}+ (\sin p\pi x-\cos p\pi x)\\
                   i[-e^{-p\pi x}-(\sin p\pi x-\cos p\pi x)] \\
0
                 \end{array}
               \right) +O(n^{-1}),
  \end{align}
 where
  \begin{equation}
    G_n(x)=\frac{1}{\omega_n^2e^{\omega_n}}\begin{pmatrix}
      \phi_{n}''(x) \\
      \mu_n\phi_n(x)\\
      \gamma\phi'_n(0)
    \end{pmatrix}^T,
  \end{equation}
$(\phi_n,\mu_n\phi_n)$ is the eigenfunction of $J_1$ corresponding to the eigenvalue $\mu_n$.

{\color{blue}Indeed, it is easily seen by definition that $J_1$ is skew-adjoint. Therefore $J_1$ generates an unitary group and the spectrum of $J_1$ is on the imaginary. Let $(f,g)\in \mathbb{H}$, solve the equation $J_1(\phi,\psi)=(f,g)$ to get
$  \psi(x)=f(x), x\in (0,1)$ and
\begin{align*}
    \phi(x)=\frac{x^3}{6}\int_0^1(1-\sigma)g(\sigma)d\sigma-\frac{x}{6}\int_0^1(\sigma^3-3\sigma^2+2\sigma)g(\sigma)d\sigma-\int_0^{x}\frac{(x-\sigma)^3}{6}g(\sigma)d\sigma.
\end{align*}
Hence $0\in\rho(J_1)$. Since by Sobolev imbedding theorem $H^4(0,1)\times H^2(0,1)$ is compactly imbedded in $H^2(0,1)\times L^2(0,1)$, $J_1^{-1}$ is impact. Therefore the eigenfunction of $J_1$ forms a orthogonal basis. Put $\frac{1}{c}=\gamma=0$ in
(\ref{fxequ3}) to derive the eigenfunction $(\phi,\mu\phi)$ corresponding to
eigenvalue $\mu=i\omega^2$ as follows
\begin{align*}
  \phi(x)=2[\sinh\tau(x-1)+\sin\tau \cosh\tau x-\cos\tau \sinh\tau x+\sin\tau(x-1)+\cos\tau x\sinh\tau-\sin\tau x\cosh\tau].
\end{align*}
In (\ref{tauchar2}), let $\frac{1}{c}=\gamma=0$ and $\tau=\omega$ to get
\begin{equation}\label{aomegachar}
  \cosh\omega\sin\omega-\sinh\omega\cos\omega=0.
\end{equation}
Similar to Step 2, we use Routh\'{e}'s theorem to derive
\begin{equation*}
  \omega_n=p\pi+O(e^{-n}), p=n+\frac{1}{4}.
\end{equation*}
Moreover, by the same procedure as in Step 2, it is easily obtained (\ref{aphixasy0}).
}

{\bf Step 5:} We show that there exist a sequence of generalized eigenfunctions of $\mathbb{A}$ which forms Riesz basis for $\mathbb{H}$;
$\mathbb{A}$ is a generator of exponentially stable $C_0$-semigroup.
To this end, {\color{blue}similar to Guo, Wang and Yang \cite{Guo2008}, we define an isometric isomorphism $\mathbb{T}:\mathbb{H}\rightarrow L^2(0,1)\times L^2(0,1)\times \mathds{R}$
by $\mathbb{T}(f,g)=(f'',g,\gamma f'(0)),(f,g)\in \mathbb{H}$.}

In fact, by (\ref{fxequasy30}) and (\ref{aphixasy0}), there exists $N>0$ such that
\begin{equation}
\begin{split}
  &\sum_{n>N}^{\infty}\bigg\|\frac{1}{\tau_n^2e^{\tau_n}}(f_n,\lambda_nf_n)-\frac{1}{\omega_n^2e^{\omega_n}}(\phi_n,\mu_n\phi_n)\bigg\|^2_{\mathbb{H}}\\
  =&{\color{blue}\sum_{n>N}^{\infty}\bigg\|\frac{1}{\tau_n^2e^{\tau_n}}\mathbb{T}(f_n,\lambda_nf_n)-\frac{1}{\omega_n^2e^{\omega_n}}\mathbb{T}(\phi_n,\mu_n\phi_n)\bigg\|^2_{L^2(0,1)\times L^2(0,1)\times \mathds{R}}}\\
  =&\sum_{n>N}^{\infty}\|F_n-G_n\|^2_{L^2(0,1)\times L^2(0,1)\times \mathds{R}}=\sum_{n>N}^{\infty}O(n^{-2})<\infty.
\end{split}
\end{equation}
The same thing is true for conjugates. Since $\mathbb{A}$ is a densely defined and discrete operator, $\mathbb{A}$ satisfies all the conditions of Theorem 1 of \cite{Guo2001}. Hence  the generalized eigenfunctions of $\mathbb{A}$ forms Riesz basis for $\mathbb{H}$.
This indicates that $\mathbb{A}$ is a generator of $C_0$-semigroup. The exponential stability is obtained by
step 3 and the asymptotic expression (\ref{ataunasy1}). The proof is therefore completed.
\end{proof}

Similar to \cite[Lemma 2.3]{Zhou2018a}, we obtain the following lemma.
\begin{lemma}\label{zestimate}
Assume that $(\widehat{z}(\cdot,0),\widehat{z}_t(\cdot,0))\in D(\mathbb{A})$. The solution of (\ref{erroerrorobserver}) satisfies $|\widehat{z}_{xxx}(1,t)|\leq Me^{-\mu t}$ for some constant $M,\; \mu>0$.
\end{lemma}

We put systems (\ref{transferl}) and (\ref{errorobserver}) together to derive the following
infinite-dimensional disturbance estimator
\begin{equation} \label{transfer}
\left\{\begin{array}{l}
l_{tt}(x,t)+l_{xxxx}(x,t)=0,\;\; x\in (0,1), \; t>0, \\
l(0,t)=l_{xx}(1,t)=0, \; \;  t\ge 0,\\
l_{xx}(0,t)=cl_{xt}(0,t)+\gamma l_x(0,t)+w_{xx}(0,t) \; \;  t\ge 0,\\
-l_{xxx}(1,t)+ml_{tt}(1,t)=u(t), \;\;  t\ge 0, \\
z_{tt}(x,t)+z_{xxxx}(x,t)=0,\;\; x\in (0,1), \; t>0, \\
z(0,t)=z_{xx}(1,t)=0,z(1,t)=l(1,t)-w(1,t) \; \;  t\ge 0,\\
z_{xx}(0,t)=c z_{xt}(0,t)+\gamma z_x(0,t), \;\;  t\ge 0.
\end{array}\right.
\end{equation}
By Lemma \ref{zestimate}, for $(\widehat{z}(\cdot,0),\widehat{z}_t(\cdot,0))\in D(\mathbb{A})$,
$F(t)=\widehat{l}_{xxx}(1,t)-m\widehat{l}_{tt}(1,t)=z_{xxx}(1,t)-mz_{tt}(1,t)-\widehat{z}_{xxx}(1,t)\approx z_{xxx}(1,t)-mz_{tt}(1,t)$, which means that the observer system (\ref{erroerrorobserver}) is a total disturbance estimator.
Moreover, since (\ref{erroerrorobserver}) decays exponentially, we obtain  $w(\cdot,t)=l(\cdot,t)-z(\cdot,t)+\widehat{z}(\cdot,t)\approx l(\cdot,t)-z(\cdot,t)$,
this implies that $l(\cdot,t)-z(\cdot,t)$ is the estimate of $w(\cdot,t)$.

Since the state feedback $u(t)=-\alpha w_t(1,t)+\beta w_{xxxt}(1,t)$ exponentially stabilizes system (\ref{beem}) without disturbance, $z_{xxx}(1,t)-mz_{tt}(1,t)-\widehat{z}_{xxx}(1,t)$ is the estimate of total disturbance $F(t)$, and $w(\cdot,t)$ is estimated by $l(\cdot,t)-z(\cdot,t)$, it is natural to design the following controller
\begin{align}\label{feedback11}
    \nonumber&u(t)=-z_{xxx}(1,t)+mz_{tt}(1,t)-\alpha[l_t(1,t)-z_t(1,t)]\\
    &+\beta [l_{xxxt}(1,t)-z_{xxxt}(1,t)].
\end{align}
With the controller (\ref{feedback11}), we derive the closed-loop system
\begin{equation} \label{perror110}
\left\{\begin{array}{l}
w_{tt}(x,t)+w_{xxxx}(x,t)=0, \\
w(0,t)=w_x(0,t)=w_{xx}(1,t)=0,\\
-w_{xxx}(1,t)+mw_{tt}(1,t)=-z_{xxx}(1,t)+mz_{tt}(1,t)\\
-\alpha(l_t(1,t)-z_t(1,t))
+\beta [l_{xxxt}(1,t)\\
-z_{xxxt}(1,t)]+F(t),\\
l_{tt}(x,t)+l_{xxxx}(x,t)=0,\;\; x\in (0,1), \\
l(0,t)=l_{xx}(1,t)=0, \\
l_{xx}(0,t)=cl_{xt}(0,t)+\gamma l_x(0,t)+w_{xx}(0,t),\\
-l_{xxx}(1,t)+ml_{tt}(1,t)=-z_{xxx}(1,t)+mz_{tt}(1,t)\\
-\alpha(l_t(1,t)-z_t(1,t))
+\beta [l_{xxxt}(1,t)-z_{xxxt}(1,t)], \\
z_{tt}(x,t)+z_{xxxx}(x,t)=0,\\
z(0,t)=z_{xx}(1,t)=0, z(1,t)=l(1,t)-w(1,t),\\
z_{xx}(0,t)=c z_{xt}(0,t)+\gamma z_x(0,t).
\end{array}\right.
\end{equation}
\begin{remark}\label{xiangdui}
In \cite{Zhou2018a,Zhou2018b}, disturbance estimator designs were only used to estimate
the total disturbance; in order to derive estimated states, the authors designed additional Luenberger state observers of the original system by compensating the total disturbance. However, we observe that the state of the original system can also be estimated by the disturbance estimator. Based on this, in control law (\ref{feedback11}), we directly used the estimated state stemmed from the disturbance estimator. This makes our control strategy more concise and energy-saving, and our method provides an idea to simplify the references \cite{Zhou2018a,Zhou2018b}.
\end{remark}

\section{Stability of the Close-loop system}

In this section, our objective is to verify that the state $(w(\cdot,t),w_t(\cdot,t))$ of the closed-loop system (\ref{perror110}) is exponentially stable while
guaranteeing the boundedness of the other variables.
To this end, we first write $w$-part of (\ref{perror110}) and (\ref{erroerrorobserver}) together to get
\begin{equation} \label{wzwanclosed}
\left\{\begin{array}{l}
w_{tt}(x,t)+w_{xxxx}(x,t)=0, \\
w(0,t)=w_x(0,t)=w_{xx}(1,t)=0, \\
-w_{xxx}(1,t)+mw_{tt}(1,t)=-\alpha w_t(1,t)\\
+\beta w_{xxxt}(1,t)-\widehat{z}_{xxx}(1,t)-\beta\widehat{z}_{xxxt}(1,t),\\
\widehat{z}_{tt}(x,t)+\widehat{z}_{xxxx}(x,t)=0,\\
\widehat{z}(0,t)=\widehat{z}(1,t)=\widehat{z}_{xx}(1,t)=0, \\
\widehat{z}_{xx}(0,t)=c \widehat{z}_{xt}(0,t)+\gamma \widehat{z}_x(0,t).
\end{array}\right.
\end{equation}

It is natural to choose the state Hilbert space of system (\ref{wzwanclosed}) as $\mathcal{H}_1=\mathbf{H}\times\mathbb{H}$.
Define the operator $\mathcal{A}_1:D(\mathcal{A}_1)\subset\mathcal{H}_1\rightarrow \mathcal{H}_1$ by
$\mathcal{A}_1(f,g,\eta,p,q)=(g,-f^{(4)},-\eta\beta^{-1}-\beta^{-1}(\alpha-m\beta^{-1})g(1),q,-p^{(4)}),\; (f,g,p,q)\in D(\mathcal{A}_1)=\{(f,g,\eta,p,q)\in (H^2_E(0,1)\bigcap H^4(0,1))\times H^2_E(0,1)\times D(\mathbb{A})|f''(1)
=0, \eta=-f'''(1)+m\beta^{-1}g(1)+p'''(1)\}$.
System (\ref{wzwanclosed}) is abstractly described by
$\frac{d}{dt}X(t)=\mathcal{A}_1X(t),$
where $X(t)=(w(\cdot,t),w_t(\cdot,t),-w_{xxx}(1,t)+m\beta^{-1}w_t(1,t)+\widehat{z}_{xxx}(1,t),\widehat{z}(\cdot,t),\widehat{z}_t(\cdot,t)).$
It is easily seen that the operator $\mathcal{A}_1$ is not dissipative in the current inner product;
it seems difficult to find an equivalent inner product in $\mathcal{X}_1$ to make $\mathcal{A}_1$ dissipative; the
multiplier method may be not effective to verify the stability of the semigroup $e^{\mathcal{A}_1t}$.
Instead, we shall use Riesz basis approach to prove the stability, the key step is to find out the complicated but important
relations (\ref{guanxi3}) and (\ref{guanxi4}) between sequences of generalized eigenfunctions.

\begin{theorem}\label{exponential00}
System (\ref{wzwanclosed}) is governed by an exponentially stable $C_0$-semigroup.
\end{theorem}

\begin{proof}\ \
It is routine to show that $\mathcal{A}_1^{-1}$ exists and is compact on $\mathcal{X}_1$, that is,
 $\mathcal{A}_1$ is a discrete operator and the spectrum consists of eigenvalues.
By the same procedure as the proof of Theorem \ref{exponentialnodisturbance}, we can obtain $\sigma(\mathcal{A}_1)=\sigma(\mathbf{A})\bigcup \sigma(\mathbb{A})$.

Next, we shall show that the generalized eigenfunction of $\mathcal{A}_1$ forms a Riesz basis for $\mathcal{X}_1$.
Let $\{\lambda_{n},\overline{\lambda}_{n}\}_{n=1}^\infty$ and $\{\lambda_{1n},\overline{\lambda}_{1n}\}_{n=1}^\infty$ be respectively the eigenvalues of $\mathbb{A}$ and $\mathbf{A}$, $\lambda_{n}=i\tau_{n}^2,\lambda_{1n}=i\tau_{1n}^2$.
Let $\{(2\tau_{n}^{-2}e^{-\tau_{n}}\phi_{n},2ie^{-\tau_{n}}\phi_{n})\}_{n=1}^\infty$ and $\{(2\tau_{1n}^{-2}e^{-\tau_{1n}}f_{n},2ie^{-\tau_{1n}}f_{n},-\frac{2i\beta^{-1}}{\lambda_{1n}+\beta^{-1}}e^{-\tau_{1n}}(\alpha-m\beta^{-1})$
$f_{n}(1))\}_{n=1}^\infty$ be the generalized eigenfunctions corresponding to $\{\lambda_{n}\}_{n=1}^\infty$ and $\{\lambda_{1n}\}_{n=1}^\infty$ such that the sequences
 $\{(2\tau_{n}^{-2}e^{-\tau_{n}}\phi_{n},2ie^{-\tau_{n}}\phi_{n})\}_{n=1}^\infty\bigcup \{(\overline{2\tau_{n}^{-2}e^{-\tau_{n}}\phi_{n}},\overline{2ie^{-\tau_{n}}\phi_{n})})$
$\}_{n=1}^\infty $ and
$\{(2\tau_{1n}^{-2}e^{-\tau_{1n}}f_{n},$
$2ie^{-\tau_{1n}}f_{n},-\frac{2i\beta^{-1}}{\lambda_{1n}+\beta^{-1}}e^{-\tau_{1n}}(\alpha-m\beta^{-1})f_{n}(1))\}_{n=1}^\infty\bigcup
\{(\overline{2\tau_{1n}^{-2}e^{-\tau_{1n}}f_{n}},\overline{2ie^{-\tau_{1n}}f_{n}},-\overline{\frac{2i\beta^{-1}}{\lambda_{1n}+\beta^{-1}}}$
$\overline{e^{-\tau_{1n}}(\alpha}$\\
$\overline{-m\beta^{-1})f_{n}(1)})\}_{n=1}^\infty$
form Riesz basis for $\mathbb{H}$ and $ \mathbf{H}$, respectively. Accordingly, $\{(0,0,0,2\tau_{n}^{-2}e^{-\tau_{n}}\phi_{n},$
$2ie^{-\tau_{n}}\phi_{n})\}_{n=1}^\infty\bigcup \{(0,0,0,\overline{2\tau_{n}^{-2}e^{-\tau_{n}}\phi_{n}},\overline{2ie^{-\tau_{n}}\phi_{n})})\}_{n=1}^\infty$
$ \bigcup\{(2\tau_{1n}^{-2}e^{-\tau_{1n}}f_{n},2ie^{-\tau_{1n}}f_{n},-\frac{2i\beta^{-1}}{\lambda_{1n}+\beta^{-1}} e^{-\tau_{1n}}$
$(\alpha-m\beta^{-1})f_{n}(1),0,0)\}_{n=1}^\infty$
$\bigcup
\{(\overline{2\tau_{1n}^{-2}e^{-\tau_{1n}}f_{n}},$
$\overline{2ie^{-\tau_{1n}}f_{n}},-\overline{\frac{2i\beta^{-1}}{\lambda_{1n}+\beta^{-1}}e^{-\tau_{1n}}}$
$\overline{(\alpha-m\beta^{-1})f_{n}(1)},0,0)\}_{n=1}^\infty$ forms a Riesz basis for $\mathbf{H}\times \mathbb{H}$, which is equivalent to that $\{(0,0,0,2\tau_{n}^{-2}e^{-\tau_{n}}\phi''_{n},2ie^{-\tau_{n}}\phi_{n},$
$2\tau_{n}^{-2}e^{-\tau_{n}}\gamma\phi'_{n}(0))\}_{n=1}^\infty\bigcup \{(0,0,0,\overline{2\tau_{n}^{-2}\phi''_{n}}$
$\overline{e^{-\tau_{n}}},\overline{2ie^{-\tau_{n}}\phi_{n}},\overline{2\tau_{n}^{-2}e^{-\tau_{n}}\gamma\phi'_{n}(0)})\}_{n=1}^\infty \bigcup\{(2\tau_{1n}^{-2}e^{-\tau_{1n}}f''_{n},2i$
$e^{-\tau_{1n}}f_{n},-\frac{2i\beta^{-1}}{\lambda_{1n}+\beta^{-1}} e^{-\tau_{1n}}(\alpha-m\beta^{-1})f_{n}(1),0,0,0)\}_{n=1}^\infty$
$\bigcup\{(\overline{2\tau_{1n}^{-2}e^{-\tau_{1n}}f''_{n}},\overline{2ie^{-\tau_{1n}}f_{n}},-\overline{\frac{2i\beta^{-1}}{\lambda_{1n}
+\beta^{-1}}e^{-\tau_{1n}}}$
$\overline{(\alpha-m\beta^{-1})f_{n}(1)},0,0,0)\}_{n=1}^\infty$  forms a Riesz basis for $\big((L^2(0,1))^2\times \mathds{C}\big)^2$.

Let $\lambda=i\tau^2\in \sigma(\mathcal{A})$ and $(2\tau^{-2}e^{-\tau}f,2ie^{-\tau}f,$
$-\frac{2i\beta^{-1}}{\lambda+\beta^{-1}}e^{-\tau}(\alpha-m\beta^{-1})f(1),2\tau^{-2}e^{-\tau}\phi,
2ie^{-\tau}\phi)$ be the corresponding eigenfunction. If $\phi=0$, then $(2\tau^{-2}e^{-\tau}f,2ie^{-\tau}f,-\frac{2i\beta^{-1}}{\lambda+\beta^{-1}}e^{-\tau}(\alpha-m\beta^{-1})f(1))$
$\neq 0$ and $\lambda\in \sigma(\mathbf{A})$.
Hence in this case the eigenvalues $\{\lambda_{1n}\}_{n=1}^\infty$ corresponds the eigenfunctions $\{(2\tau_{1n}^{-2}e^{-\tau_{1n}}f_n,2ie^{-\tau_{1n}}f_n,-\frac{2i\beta^{-1}}{\lambda_{1n}+\beta^{-1}}e^{-\tau_{1n}}(\alpha-m\beta^{-1})$
$f_n(1),0,0)\}_{n=1}^\infty$.

If $\phi\neq 0$, then $\lambda\in \sigma(\mathbb{A})$. The eigenvalues $\{\lambda_{n}\}_{n=1}^\infty$ corresponds
the eigenfunction $(2\tau_{n}^{-2}e^{-\tau_{n}}\phi_{n},$
$2ie^{-\tau_n}\phi_{n})\}_{n=1}^\infty$ of $\mathbb{H}$, where
\begin{align}
 &\phi_{n}=\sinh\tau_{n}(x-1)+\sin\tau_{n} \cosh\tau_{n} x-\cos\tau_{n}\sinh\tau_{n} x+\frac{2\tau_{n}\sin\tau_{n}}{ic\tau_{n}^2+\gamma}\\
    &\sinh\tau_{n} x+\sin\tau_{n}(x-1)+\cos\tau_{n} x \sinh\tau_{n}-\sin\tau_{n} x \cosh\tau_{n}-\frac{2\tau_{n} \sinh\tau_{n}}{ic\tau_{n}^2+\gamma}\sin\tau_{n} x.
\end{align}
Then there holds
$(1+i\beta\tau^2)f'''_{n}(1)=2i\tau_n^2P_n,$
where
\begin{align}
  P_n=(\beta\tau_n^2-i)\tau_n\bigg[\sinh\tau_n\sin\tau_n +\frac{\tau(\sinh\tau_n\cos\tau_n
+\cosh\tau_n\sin\tau_n)}{i\alpha\tau^2+\beta}\bigg].
\end{align}
Denote by $(2\tau_n^{-2}e^{-\tau}f_{1n},2ie^{-\tau_n}f_{1n},-\frac{2i\beta^{-1}}{\lambda+\beta^{-1}}e^{-\tau_n}(\alpha-m\beta^{-1})f(1),2\tau_n^{-2}e^{-\tau_n}\phi_n,
2ie^{-\tau_n}\phi_n)$ the eigenfunction of $\mathcal{A}_1$ corresponding to $\lambda_{n}$. Then we have
\begin{align}\label{F21}
    \left\{
      \begin{array}{ll}
        f^{(4)}_{1n}(x)-\tau^4_{n}f_{1n}(x)=0, \\
        f_{1n}(0)=f_n'(0)=f''_{1n}(1)=0,  \\
       (1+i\beta\tau^2)f'''_{1n}(1)=(i\alpha\tau_n^2-m\tau_n^4)f_{1n}(1)+2i\tau_n^2 P_n,
      \end{array}
    \right.
\end{align}
The solution of (\ref{F21}) is as follows
\begin{align*}
   f_{n}(x)=b_{11}(\cosh\tau_{n}x-\cos\tau_{n}x)+b_{12}(\sinh\tau_{n}x-\sin\tau_{n}x),
\end{align*}
where
$b_{11}=-\frac{iP_n(\sinh\tau_n+\sin\tau_n)}{U_n},
        b_{12}=\frac{iP_n(\cosh\tau_n+\cos\tau_n)}{U_n},$
$U_n=\tau_n(1+\cosh\tau_n\cos\tau_n)(1+i\beta\tau_n^2)-(m\tau_n-i\alpha)(\cosh\tau_n\sin\tau_n-\sinh\tau_n\cos\tau_n).$
Hence we have
\begin{align*}
    \left\{
      \begin{array}{ll}
        2\tau_{n}^{-2}e^{-\tau_{n}}f''_{1n}(x)=-e^{-p\pi x}-\cos p\pi x\\
        +\sin p\pi x+O(n^{-1})\\
        2\tau_{n}^{-2}e^{-\tau_{n}}[\lambda_nf_{1n}(x)]=i\big[-e^{-p\pi x}+\cos p\pi x\\
        -\sin p\pi x\big]+O(n^{-1})\\
        -\frac{2i\beta^{-1}}{\lambda_n+\beta^{-1}}e^{-\tau_n}(\alpha-m\beta^{-1})f_{1n}(1)=O(n^{-1}),
      \end{array}
    \right.
\end{align*}
where $p=n+1/4$.
Denote $Q=\left(
            \begin{array}{cc}
              I_3 & J \\
              0_{3\times 3} & I_3 \\
            \end{array}
          \right)
$ with $J=\left(
                              \begin{array}{ccccc}
                                1 & 0& 0 \\
                                 0 &1& 0  \\
                                 0 &0&0 \\
                              \end{array}
                            \right)^T.$
Then $Q$ is a bounded linear operator and it has bounded inverse.
Moreover, we obtain the following relations
\begin{align}
\nonumber    &(2\tau_{1n}^{-2}e^{-\tau_{1n}}f''_{n},2ie^{-\tau_{1n}}f_{n},-\frac{2i\beta^{-1}}{\lambda_{1n}+\beta^{-1}}e^{-\tau_{1n}}(\alpha-m\beta^{-1})\cdot\\
\nonumber&f_{n}(1),0,0,0)^T=Q(2\tau_{1n}^{-2}e^{-\tau_{1n}}f''_{n},2ie^{-\tau_{1n}}f_{n},\\
\label{guanxi3}&-\frac{2i\beta^{-1}}{\lambda_{1n}+\beta^{-1}}e^{-\tau_{1n}}(\alpha-m\beta^{-1})f_{n}(1),0,0,0)^T,\\
 \nonumber    &(2\tau_n^{-2}e^{-\tau_n}f''_{1n},2ie^{-\tau_n}f_{1n},-\frac{2i\beta^{-1}}{\lambda_n+\beta^{-1}}e^{-\tau_n}(\alpha\\
 \nonumber&-m\beta^{-1})f_{1n}(1),
 2\tau_n^{-2}e^{-\tau}\phi''_n,
2ie^{-\tau}\phi_n,2\tau_n^{-2}e^{-\tau}\gamma\phi'_n(0))\\
\nonumber&=Q(0,0,0,2\tau_n^{-2}e^{-\tau}\phi''_n,
2ie^{-\tau}\phi_n,2\tau_n^{-2}e^{-\tau}\gamma\phi'_n(0))\\
\label{guanxi4}&+O(n^{-1}).
\end{align}
The same thing is true for conjugates. Then, by Bari's theorem the sequence $\{(2\tau_{1n}^{-2}e^{-\tau_{1n}}f''_{n},2ie^{-\tau_{1n}}f_{n},$
$-\frac{2i\beta^{-1}}{\lambda_{1n}+\beta^{-1}}e^{-\tau_{1n}}(\alpha-m\beta^{-1})f_{n}(1),0,0,0)\}_{n=-\infty}^\infty \bigcup\{(2\tau_n^{-2}$
$ e^{-\tau_n}f''_{1n},2ie^{-\tau_n}f_{1n},-\frac{2i\beta^{-1}}{\lambda_n+\beta^{-1}}e^{-\tau_n}(\alpha-m\beta^{-1})$
$f_{1n}(1),2\tau_n^{-2}$
$e^{-\tau}\phi''_n,
2ie^{-\tau}\phi_n,2\tau_n^{-2}e^{-\tau}\gamma\phi'_n(0))\}_{n=-\infty}^\infty\bigcup
\{(\overline{2\tau_{1n}^{-2}e^{-\tau_{1n}}f''_{n}},$
$\overline{2ie^{-\tau_{1n}}f_{n}},-\overline{\frac{2i\beta^{-1}}{\lambda_{1n}
+\beta^{-1}}e^{-\tau_{1n}}}(\alpha-m\beta^{-1})\overline{f_{n}(1)},0,0,0)$
$\}_{n=-\infty}^\infty \bigcup \{(\overline{2\tau_n^{-2}e^{-\tau_n}f''_{1n}},\overline{2ie^{-\tau_n}f_{1n}},-\overline{\frac{2i\beta^{-1}}{\lambda_n+\beta^{-1}}e^{-\tau_n}(\alpha-m}$
$\overline{\beta^{-1})f_{1n}(1)},\overline{2\tau_n^{-2}e^{-\tau}\phi''_n},$
$\overline{2ie^{-\tau}\phi_n},\overline{2\tau_n^{-2}e^{-\tau}\gamma\phi'_n(0)})\}_{n=-\infty}^\infty$
forms Riesz basis for $\left(\left(L^2(0,1)\right)^2\times \mathds{C}\right)^2$,
which is equivalent to that $\{(2\tau_{1n}^{-2}e^{-\tau_{1n}}f_{n},2ie^{-\tau_{1n}}f_{n},-\frac{2i\beta^{-1}}{\lambda_{1n}+\beta^{-1}}e^{-\tau_{1n}}(\alpha-m\beta^{-1})f_{n}(1),0,0)
\}_{n=-\infty}^\infty \bigcup \{(2\tau_n^{-2}e^{-\tau_n} f_{1n},$
$ 2ie^{-\tau_n}f_{1n},-e^{-\tau_n}\frac{2i\beta^{-1}}{\lambda_n+\beta^{-1}}(\alpha-m\beta^{-1})f_{1n}(1),2\tau_n^{-2}e^{-\tau}\phi_n,
2ie^{-\tau}\phi_n)\}_{n=-\infty}^\infty\bigcup\{(\overline{2\tau_{1n}^{-2}e^{-\tau_{1n}}f_{n}},\overline{2ie^{-\tau_{1n}}f_{n}},$
$ -\overline{\frac{2i\beta^{-1}}{\lambda_{1n}+\beta^{-1}}
e^{-\tau_{1n}}(\alpha-m\beta^{-1})f_{n}(1)},0,0)\}_{n=-\infty}^\infty \bigcup \{(\overline{2\tau_n^{-2}e^{-\tau_n}f_{1n}},\overline{2ie^{-\tau_n}f_{1n}},-\overline{\frac{2i\beta^{-1}}{\lambda_n
+\beta^{-1}}e^{-\tau_n}(\alpha-m\beta^{-1})}$
$\overline{f_{1n}(1)},\overline{2\tau_n^{-2}e^{-\tau}\phi_n},
\overline{2ie^{-\tau}\phi_n)}\}_{n=-\infty}^\infty$
forms Riesz basis for $\mathbf{H}\times \mathbb{H}$.

The semigroup generation of the operator $\mathcal{A}_1$ is directly derived by the Riesz basis property. Moreover, the Riesz basis property implies that the property of spectrum-determined growth condition holds for $\mathcal{A}_1$. Since $ \mathbf{A}$ and $\mathbb{A}$ generate exponentially stable $C_0$-semigroups and spectrum-determined growth condition holds, $\sup\{{\rm Re} \lambda:\lambda\in \sigma(\mathcal{A}_1)\}=\sup\{{\rm Re} \lambda:\lambda\in \sigma(\mathbf{A})\bigcup \sigma(\mathbb{A})\}<0$. Therefore $e^{\mathcal{A}t}$ is exponentially stable and the proof is
completed.
\end{proof}

\begin{remark}
In the proof of Theorem \ref{exponential00}, although it forms a Riesz basis
for $\mathcal{H}_1$, not all the components of
the sequence $\{(0,0,0,2\tau_{n}^{-2}e^{-\tau_{n}}\phi_{n},2ie^{-\tau_{n}}\phi_{n})\}_{n=1}^\infty\bigcup \{(0,0,0,\overline{2\tau_{n}^{-2}e^{-\tau_{n}}\phi_{n}},\overline{2ie^{-\tau_{n}}\phi_{n})})$
$\}_{n=1}^\infty\bigcup\{(2\tau_{1n}^{-2}e^{-\tau_{1n}}f_{n},2ie^{-\tau_{1n}}f_{n},
-\frac{2i\beta^{-1}}{\lambda_{1n}+\beta^{-1}} $
$e^{-\tau_{1n}}(\alpha-m\beta^{-1})f_{n}(1),0,0)\}_{n=1}^\infty\bigcup \{(\overline{2\tau_{1n}^{-2}e^{-\tau_{1n}}f_{n}},-2i\overline{e^{-\tau_{1n}}}$
$\overline{f_{n}},-\overline{\frac{2i\beta^{-1}}{\lambda_{1n}+\beta^{-1}}e^{-\tau_{1n}}(\alpha-m\beta^{-1})f_{n}(1)},0,0)\}_{n=1}^\infty$  are the generalized eigenfunctions of $\mathcal{A}$. To overcome this difficulty, the key step is to find out the relations (\ref{guanxi3}) and (\ref{guanxi4}).
\end{remark}

By Theorem \ref{exponential00}, it follows that, there exist two positive constants $M_{\mathcal{A}_1}$ and $\omega_{\mathcal{A}_1}$ such that
\begin{align*}
  \|e^{\mathcal{A}_1t}\|_{\mathcal{H}_1}\leq M_{\mathcal{A}_1}e^{\omega_{\mathcal{A}_1}t},\;\; t\geq 0.
\end{align*}

\begin{theorem}\label{maintheorem}
Assume that $d\in L^\infty(0,\infty)$ (or $d\in L^2(0,\infty)$) and $v:H^2(0,1)\times L^2(0,1)\rightarrow \mathds{R}$ is continuous.
For any initial value $(w(\cdot,0),w_t(\cdot,0),-l_{xxx}(1,t)+z_{xxx}(1,t)+m\beta^{-1}w_t(1,t),l(\cdot,0),l_t(\cdot,0),$
$z(\cdot,0),z_t(\cdot,0)) \in \mathbf{H}$
$\times \left(H_L^2(0,1)\times L^2(0,1)\right)^2$
with $z(1,0)=l(1,0)-w(1,0)$, there exists a unique solution to system (\ref{perror110}) such that
$(w(\cdot,t),w_t(\cdot,t),l(\cdot,t),l_t(\cdot,t),z(\cdot,t),z_t(\cdot,t))\in C(0,\infty;H^2_E(0,1)\times L^2(0,1)\times (H^2_L(0,1)\times L^2(0,1))^2)$ satisfying properties $z(1,t)=l(1,t)-w(1,t)$,
\begin{align}
  \nonumber&\int_0^1(|w_t(x,t)|^2+|w_{xx}(x,t)|^2)dx+|-l_{xxx}(1,t)\\
 \label{P1} &+m\beta^{-1} w_t(1,t)+z_{xxx}(1,t)|^2\leq M_2e^{-\gamma_2 t},\\
   \nonumber &\sup_{t\geq 0} \bigg[\int_0^1(|l_t(x,t)|^2+|l_{xx}(x,t)|^2+|z_t(x,t)|^2\\
   \nonumber&+|z_{xx}(x,t)|^2)dx+\gamma(|l_x(0,t)|^2+|z_x(0,t)|^2)\\
\label{p2}&+|l_t(1,t)-w_t(1,t)|^2)\bigg]<\infty,
\end{align}
where $M_3$ and $\gamma_3$ are two positive constants.
If in addition $v(0,0)=0$ and $d\in L^2(0,\infty)$, then
\begin{align}\label{P3}
 \nonumber&\lim_{t\rightarrow +\infty} \bigg[\int_0^1(|l_t(x,t)|^2+|l_{xx}(x,t)|^2+|z_t(x,t)|^2\\
&+|z_{xx}(x,t)|^2)dx+\gamma(|l_x(0,t)|^2+|z_x(0,t)|^2)+|l_t(1,t)-w_t(1,t)|^2\bigg]=0.
\end{align}
\end{theorem}
\begin{proof}\ \
Fix the initial value $(w(\cdot,0),w_t(\cdot,0),-l_{xxx}(1,t)$
$+z_{xxx}(1,t)+m\beta^{-1}w_t(1,t),l(\cdot,0),l_t(\cdot,0),z(\cdot,$
$0),z_t(\cdot,0)) \in \mathbf{H}$
$\times \left(H_L^2(0,1)\times L^2(0,1)\right)^2$
with $z(1,0)=l(1,0)-w(1,0)$.
This indicates that $(\widehat{z}(\cdot,t),\widehat{z}_t(\cdot,t))=e^{\mathbb{A}t}(\widehat{z}(\cdot,0),\widehat{z}_t(\cdot,0))\in \mathbb{H}$,
thereby $z(1,t)=l(1,t)-w(1,t)$.
By Theorem \ref{exponential00}, it follows that
\begin{align}
   \nonumber &\int_0^1(|w_t(x,t)|^2+|w_{xx}(x,t)|^2)dx+|-l_{xxx}(1,t)+m\beta^{-1} w_t(1,t)+z_{xxx}(1,t)|^2\\
\nonumber&=\int_0^1(|w_t(x,t)|^2+|w_{xx}(x,t)|^2)dx+|-w_{xxx}(1,t)+m\beta^{-1} w_t(1,t)+\widehat{z}_{xxx}(1,t)|^2\\
\nonumber&\leq \|(w(\cdot,t),w_t(\cdot,t),-w_{xxx}(1,t)+m\beta^{-1}w_t(1,t)+\widehat{z}_{xxx}(1,t),\widehat{z}(\cdot,t),\widehat{z}_t(\cdot,t)) \|_{\mathcal{H}_1}^2\\
\nonumber&\leq M^2_{\mathcal{A}_1}e^{-2\omega_{\mathcal{A}_1}t}\|(w(\cdot,0),w_t(\cdot,0),-w_{xxx}(1,0)+m\beta^{-1}w_t(1,0)
+\widehat{z}_{xxx}(1,0),\widehat{z}(\cdot,0),\widehat{z}_t(\cdot,0)) \|_{\mathcal{H}_1}^2\\
&= M_2e^{-\gamma_2 t},
\end{align}
where $M_{\mathcal{A}_1}$ and $\omega_{\mathcal{A}_1}$ are two positive constants,  $\gamma_2=2\omega_{\mathcal{A}_1}$ and $M_2=M^2_{\mathcal{A}_1}
\|(w(\cdot,0),w_t(\cdot,0),$
$-l_{xxx}(1,0)+m\beta^{-1}w_t(1,0)
+z_{xxx}(1,0),z(\cdot,0)-l(\cdot,0)+w(\cdot,0),z_t(\cdot,0)-l_t(\cdot,0)+w_t(\cdot,0)) \|^2_{\mathcal{H}_1}$.
This is (\ref{P1}).

The continuity and exponential stability of $(w(\cdot,t),w_t(\cdot,t))$ on $H_E^2(0,1)\times L^2(0,1)$ is therefore obtained.
Since $f$ is continuity, we derive $v(w(\cdot,t),w_t(\cdot,t))\in L^\infty(0,\infty)$.
By Lemma \ref{admissible}, we obtain $\big(\widehat{l}(\cdot,t),\widehat{l}_t(\cdot,t),m\widehat{l}_t(1,t)\big)\in C(0,\infty;\mathbf{H}_2)$
and $\sup_{t\geq 0}\|\big(\widehat{l}(\cdot,t),\widehat{l}_t(\cdot,t),$
$m\widehat{l}_t(1,t)\big)\|_{\mathbf{H}_2}<+\infty$.
Then, we have $\sup_{t\geq 0}|l_t(1,t)-w_t(1,t)|<\infty$,
\begin{align}
  \nonumber & \int_0^1[|l_t(x,t)|^2+|l_{xx}(x,t)|^2]dx+\gamma|l_x(0,t)|^2\\
\nonumber&\leq 2\int_0^1[|\widehat{l}_t(x,t)|^2+|\widehat{l}_{xx}(x,t)|^2]dx+2\gamma|\widehat{l}_x(0,t)|^2+2\int_0^1[|w_t(x,t)|^2+|w_{xx}(x,t)|^2]dx\\
\nonumber&\leq 2\sup_{t\geq 0}\|\big(\widehat{l}(\cdot,t),\widehat{l}_t(\cdot,t),m\widehat{l}_t(1,t)\big)\|^2_{\mathbf{H}_2}
\\
\nonumber&+2\|(w(\cdot,t),w_t(\cdot,t),-w_{xxx}(1,t)+m\beta^{-1}w_t(1,t)+\widehat{z}_{xxx}(1,t),\widehat{z}(\cdot,t),\widehat{z}_t(\cdot,t)) \|_{\mathcal{H}_1}^2\\
&<\infty,
\end{align}
and
\begin{align*}
   & \int_0^1[|z_t(x,t)|^2+|z_{xx}(x,t)|^2]dx+\gamma|z_x(0,t)|^2\\
&\leq 2\int_0^1[|\widehat{l}_t(x,t)|^2+|\widehat{l}_{xx}(x,t)|^2]dx+2\gamma|\widehat{l}_x(0,t)|^2\\
&+2\int_0^1[|\widehat{z}_t(x,t)|^2+|\widehat{z}_{xx}(x,t)|^2]dx+2\gamma|\widehat{z}_x(0,t)|^2\\
&\leq 2\sup_{t\geq 0}\|\big(\widehat{l}(\cdot,t),\widehat{l}_t(\cdot,t),m\widehat{l}_t(1,t)\big)\|^2_{\mathbf{H}_2}
+2\|(w(\cdot,t),\\
&w_t(\cdot,t),-w_{xxx}(1,t)+m\beta^{-1}w_t(1,t)+\widehat{z}_{xxx}(1,t),\\
&\widehat{z}(\cdot,t),\widehat{z}_t(\cdot,t)) \|_{\mathcal{H}_1}^2<\infty.
\end{align*}
Hence (\ref{perror110}) has a unique solution and $(w(\cdot,t),w_t(\cdot,t),l(\cdot,t),$
$l_t(\cdot,t),z(\cdot,t),z_t(\cdot,t))\in C(0,\infty;$
$H^2_E(0,1)\times L^2(0,1)\times (H^2_L(0,1)\times L^2(0,1))^2)$ and (\ref{p2}) is derived.

In the case $v(0,0)=0$, the combination of the continuity of $f$ and the exponential stability of $(w(\cdot,t),w_t(\cdot,t))$ implies that $\lim_{t\rightarrow \infty}v(w(\cdot,t),w_t(\cdot,t))=0$.
The combination of the assumption $d\in L^2(0,\infty)$ and \cite[Lemma A.1]{Zhou2018a} indicates that
$\lim_{t\rightarrow\infty}\|(\widehat{l}(\cdot,t),\widehat{l}_t(\cdot,t),m\widehat{l}_t(1,t))\|_{\mathbb{H}_1}=0.$
Hence the claim (\ref{P3}) holds. The proof is completed.
\end{proof}

\begin{remark}
Theorem \ref{maintheorem} not only gives the exponential stability of $(w(\cdot,t),w_t(\cdot,t))$ and
boundednesses of $(l(\cdot,t),$
$l_t(\cdot,t),z(\cdot,t),z_t(\cdot,t))$, but also the exponential stability of
$-l_{xxx}(1,t)+m\beta^{-1} w_t(1,t)+z_{xxx}(1,t)$ and the boundedness of $l_t(1,t)-w_t(1,t)$.
The reason is that
$\frac{d}{d t}[-l_{xxx}(1,t)+m\beta^{-1} w_t(1,t)+z_{xxx}(1,t)]
  =\frac{d}{d t}[-w_{xxx}(1,t)+m\beta^{-1} w_t(1,t)+\widehat{z}_{xxx}(1,t)]
  =\beta^{-1}[w_{xxx}(1,t)-\alpha w_t(1,t)-\widehat{z}_{xxx}(1,t)]
  =\beta^{-1}[-\alpha w_t(1,t)-z_{xxx}(1,t)+l_{xxx}(1,t)]$
and
$\frac{d}{d t}[m(l_t(1,t)-w_t(1,t))]=\widehat{l}_{xxx}(1,t)-F(t)=l_{xxx}(1,t)-w_{xxx}(1,t)+v(w(\cdot,t),w_t(\cdot,t))+d(t)$ are regarded as dynamic boundary conditions of the closed-loop system (\ref{perror110}).
\end{remark}

\begin{remark}
The linear (\ref{wzwanclosed}) is essentially important in the proof of Theorem \ref{maintheorem},
because by virtue of semigroup theorem, the existence, continuity and exponential stability of the solution of (\ref{beem}) can
be obtained without the global Lipsichitz condition. This implies that $(w(\cdot,t),w_t(\cdot,t))$ is continuous and bounded. By Lemma \ref{admissible},
the existence and boundedness of $(\widehat{p}(\cdot,t),\widehat{p}_t(\cdot,t))$ are thereby derived by viewing $v(w(\cdot,t),w_t(\cdot,t))+d(t)$ as the boundary input.
\end{remark}

\begin{remark}
In Theorem \ref{maintheorem}, the differences and improvements with respect to the results in \cite{Li2017} mainly lie in that, 1) the internal uncertainty is taken into consideration, while \cite{Li2017} just studies the case of $v(w,w_t)=0$, 2) only ``low order'' measurements $w(1,t)$ and $w_{xx}(0,t)$ are adopted, while in \cite{Li2017} the velocity $w_{t}(1,t)$ as well as the high order collocated feedback $w_{xxxt}(1,t)$ were used, 3) we consider arbitrary $m,\alpha,\beta>0$ while \cite{Li2017} just solved the special case $\alpha,\beta>0$ and $m=\alpha\beta$.
\end{remark}

\section{Numerical simulation}

In order to illustrate the effectiveness of the proposed feedback control, we present in this section some numerical simulations for the closed-loop system (\ref{perror110}). Finite difference scheme is adopted and the numerical results are programmed in Matlab.
The steps of time and the space step are chosen as $1/2000$ and $1/10$, respectively. We take the nonlinear internal uncertainty
$v(w(\cdot,t),w_t(\cdot,t)) = \cos(w(1,t))$ and the external disturbance $d(t)=\sin(3t)$. We chose the tip mass $m=5$.
The turning parameters are taken as $c=\alpha=1,\beta=\gamma=2,$  and the initial values are chosen by
\begin{align*}
  \left\{
    \begin{array}{ll}
      w(x,0)=x^2, w_t(x,0)=0,  \\
      l(x,0)=x^3, l_t(x,0)=0,\\
      z(x,0)=2x, z_t(x,0)=0.
    \end{array}
  \right.
\end{align*}

Fig.1, Fig.3 and Fig. 5 are the displacements $w(x,t)$, $l(x,t)$ and $z(x,t)$, respectively.
Fig.2, Fig.4 and Fig. 6 show the velocities $w_t(x,t)$, $l_t(x,t)$ and $z_t(x,t)$, respectively.
Fig.7 is $\eta(t)=m\beta^{-1}w_t(1,t)-w_{xxx}(1,t)+\widehat{z}_{xxx}(1,t)]$, and
Fig.8 is $\phi(t)=m[l_t(1,t)-w_t(1,t)]$.
Fig.9 shows the time response of the input, and Fig. 10 displays time response of $w_t(1,t)$.
One can see from the simulations that $(w,w_t)$, $\eta$ and $w_t(1,t)$ decays rapidly; $(l,l_t)$, $(z,z_t)$,$\phi(t)=m[l(1,t)-w(1,t)]$ and $u(t)$ are bounded.

\begin{figure}
\begin{minipage}[t]{0.48\linewidth}
\centering     
\includegraphics[height=5cm,width=8cm]{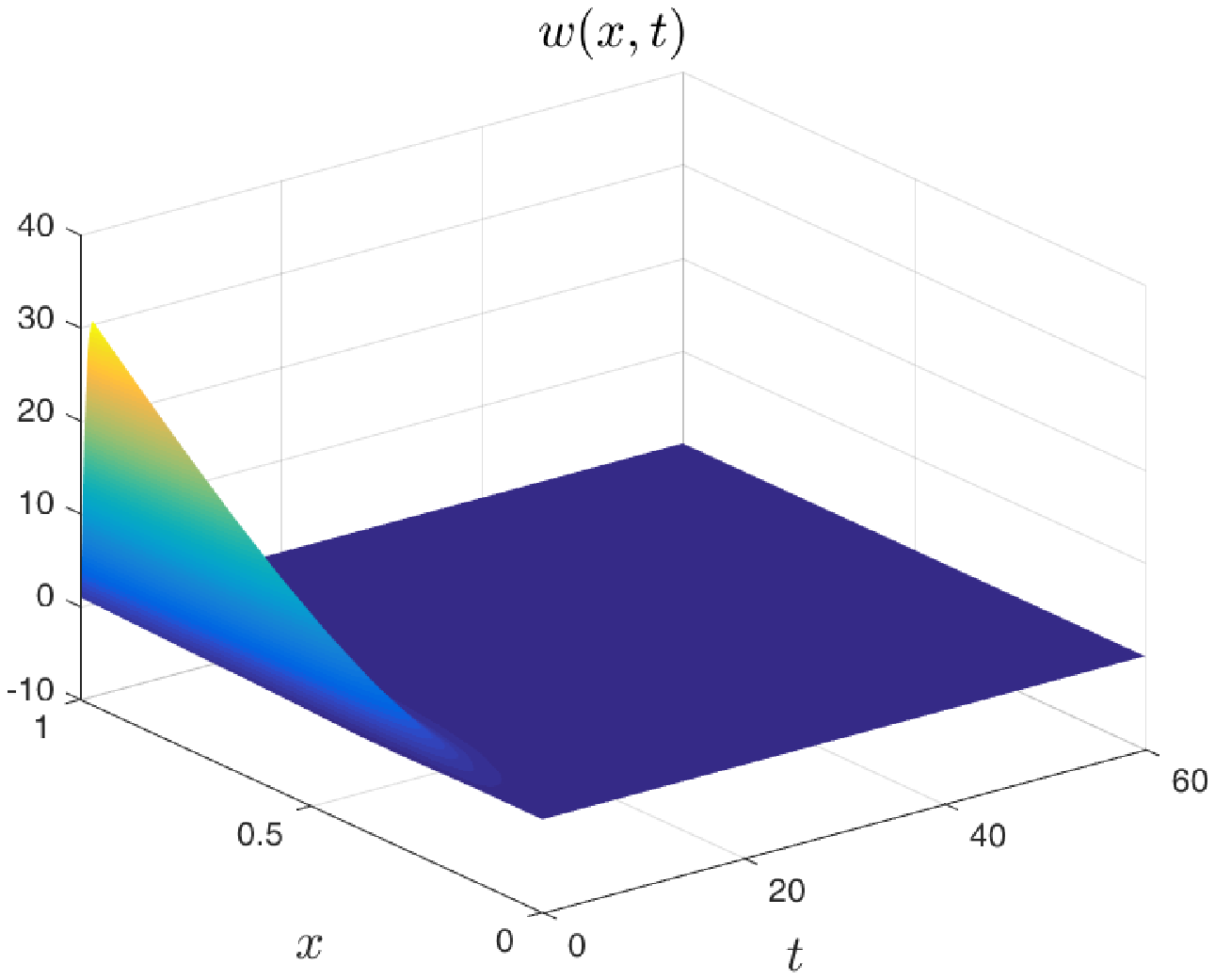}
\caption{The state $w(x,t)$.}
\label{1}
\end{minipage}
\hfill
\begin{minipage}[t]{0.48\linewidth}
\centering
\includegraphics[height=5cm,width=8.cm]{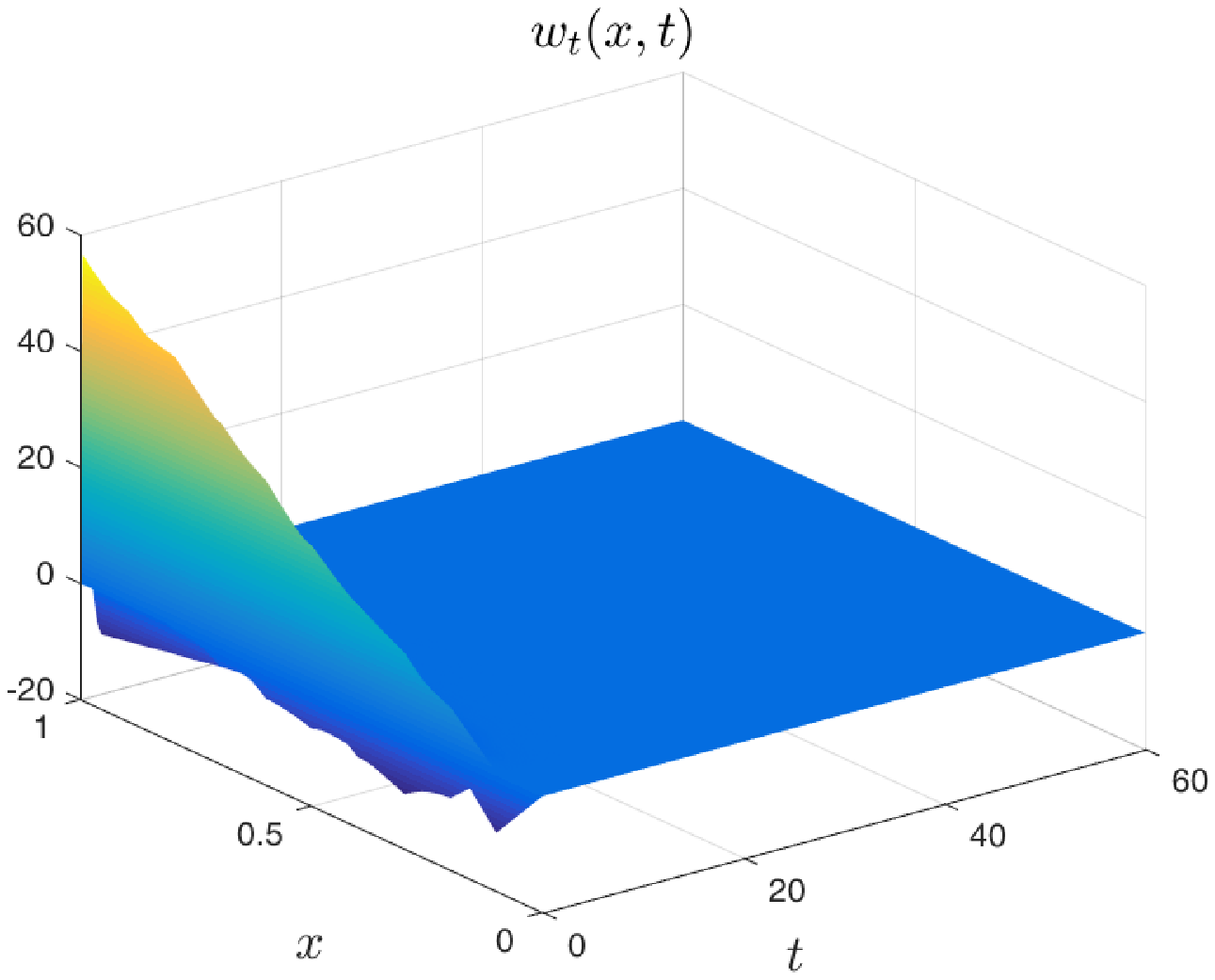}
\caption{The state $w_t(x,t)$.}
\label{2}
\end{minipage}
\end{figure}

\begin{figure}
\begin{minipage}[t]{0.48\linewidth}
\centering     
\includegraphics[height=5cm,width=8cm]{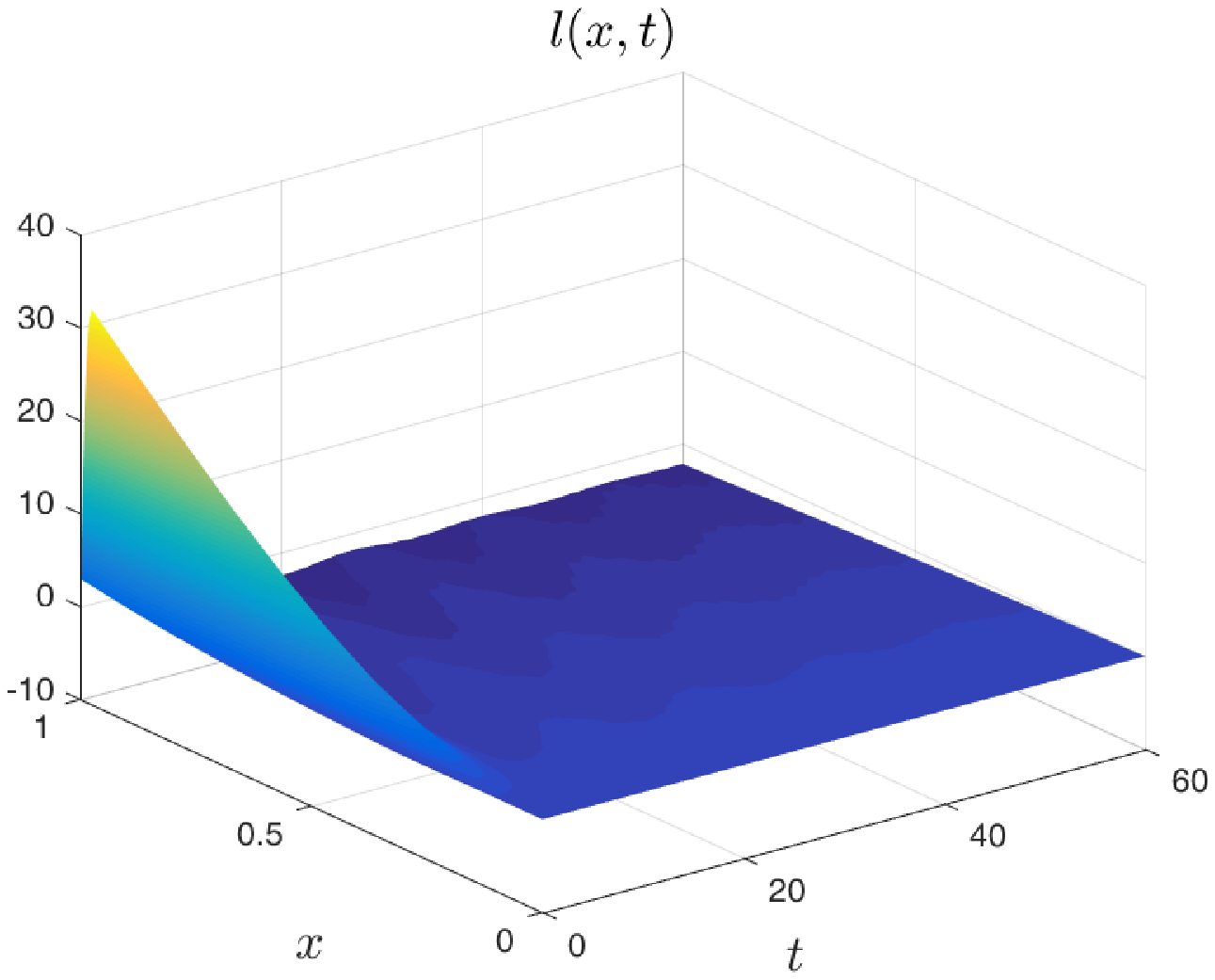}
\caption{The state $l(x,t)$.}
\label{3}
\end{minipage}
\hfill
\begin{minipage}[t]{0.48\linewidth}
\centering
\includegraphics[height=5cm,width=8cm]{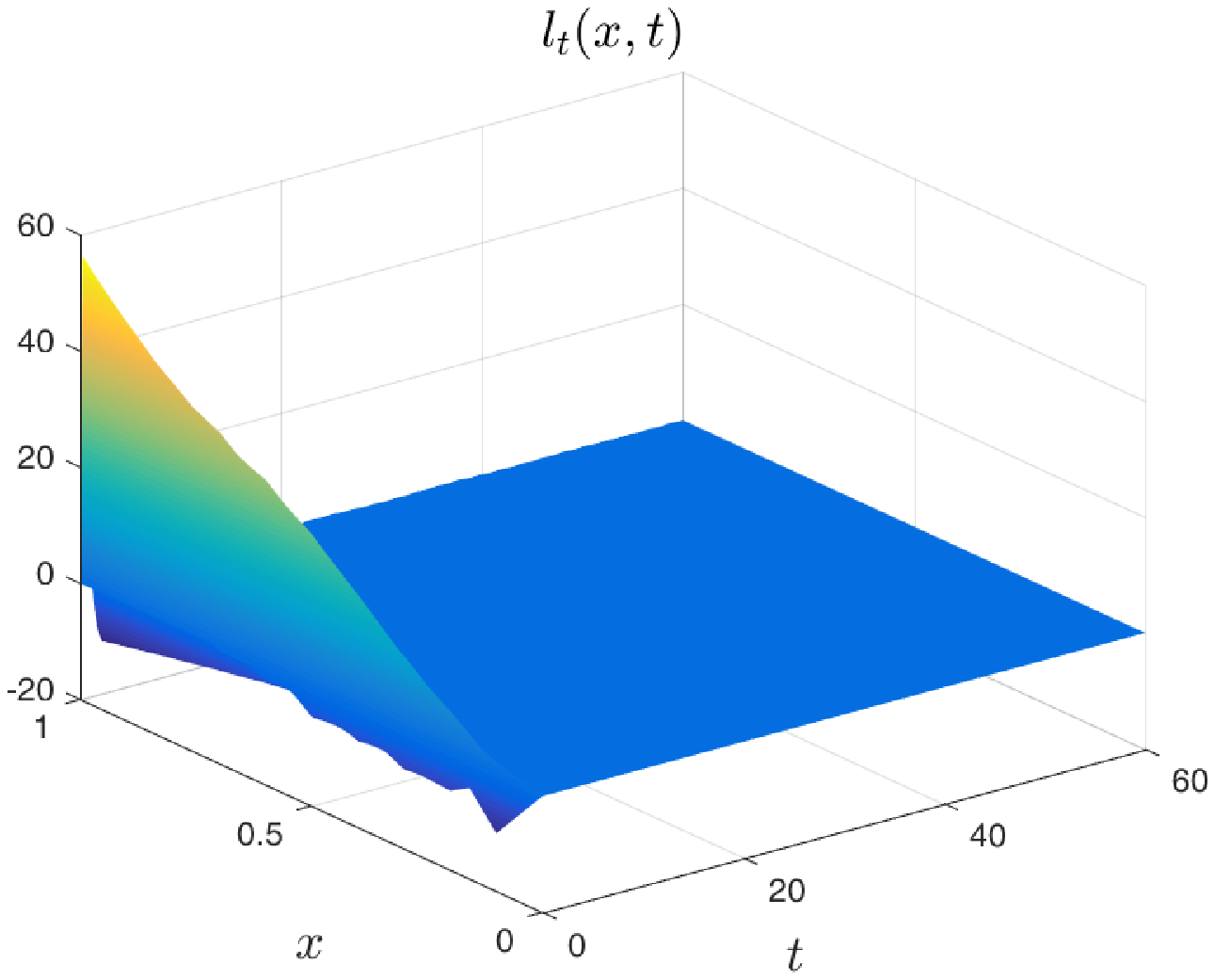}
\caption{The state $l_t(x,t)$.}
\label{4}
\end{minipage}
\end{figure}

%
%
%
%

%
%
%

\begin{figure}
\begin{minipage}[t]{0.48\linewidth}
\centering     
\includegraphics[height=5cm,width=8cm]{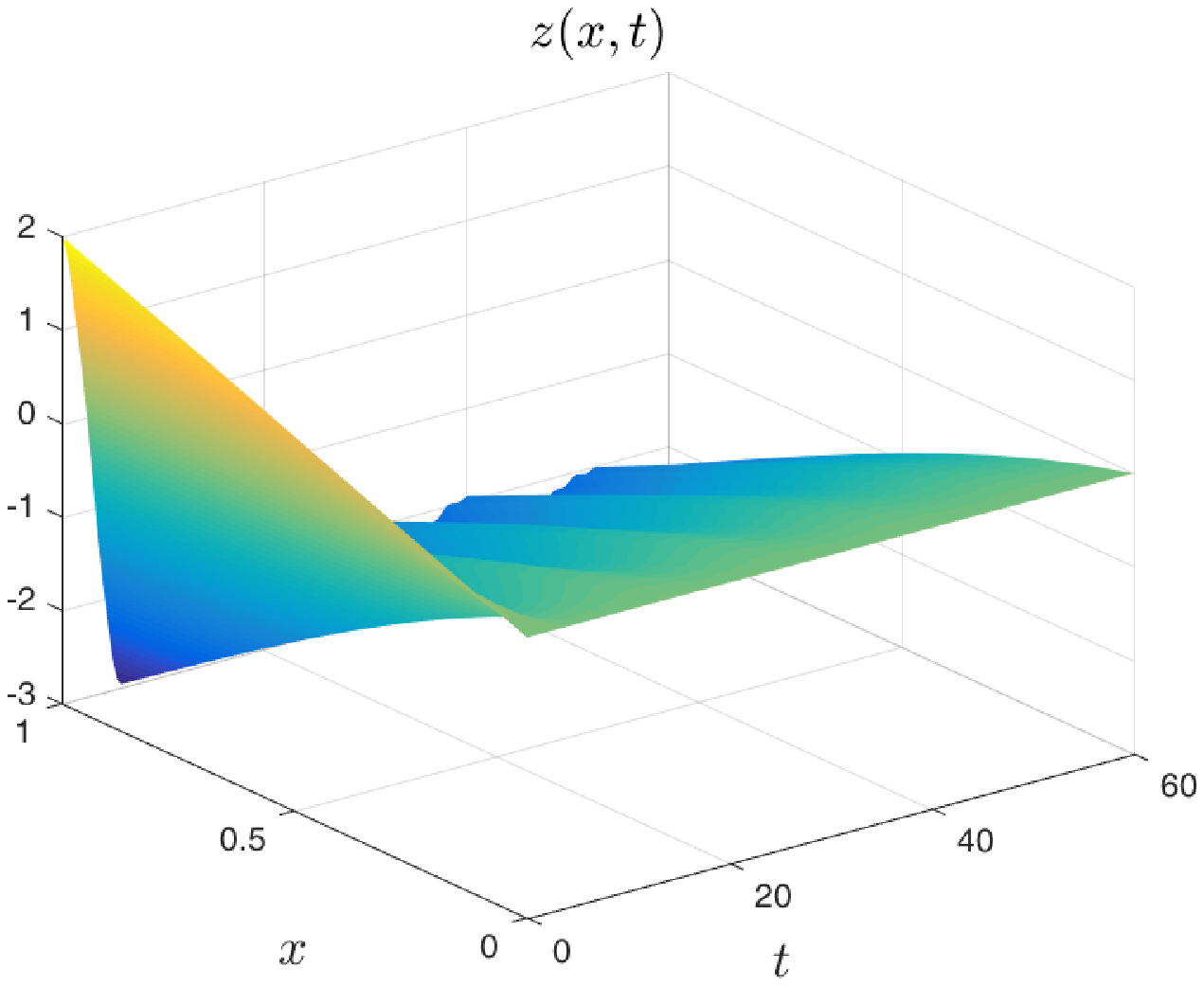}
\caption{The state $z(x,t)$.}
\label{5}
\end{minipage}
\hfill
\begin{minipage}[t]{0.48\linewidth}
\centering
\includegraphics[height=5cm,width=8cm]{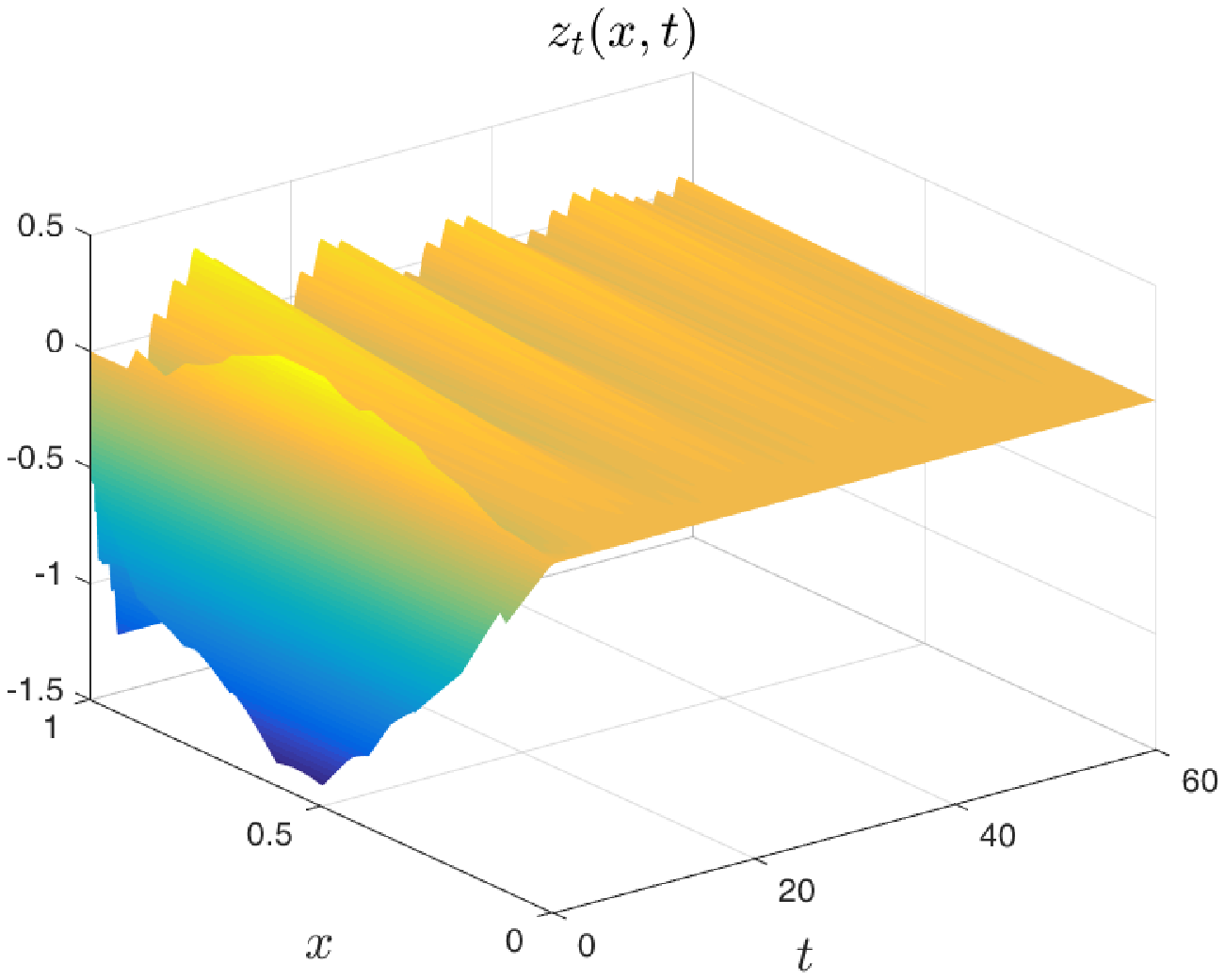}
\caption{The state $z_t(x,t)$.}
\label{6}
\end{minipage}
\end{figure}

%
%
%
%

\begin{figure}
\begin{minipage}[t]{0.48\linewidth}
\centering     
\includegraphics[height=5cm,width=8cm]{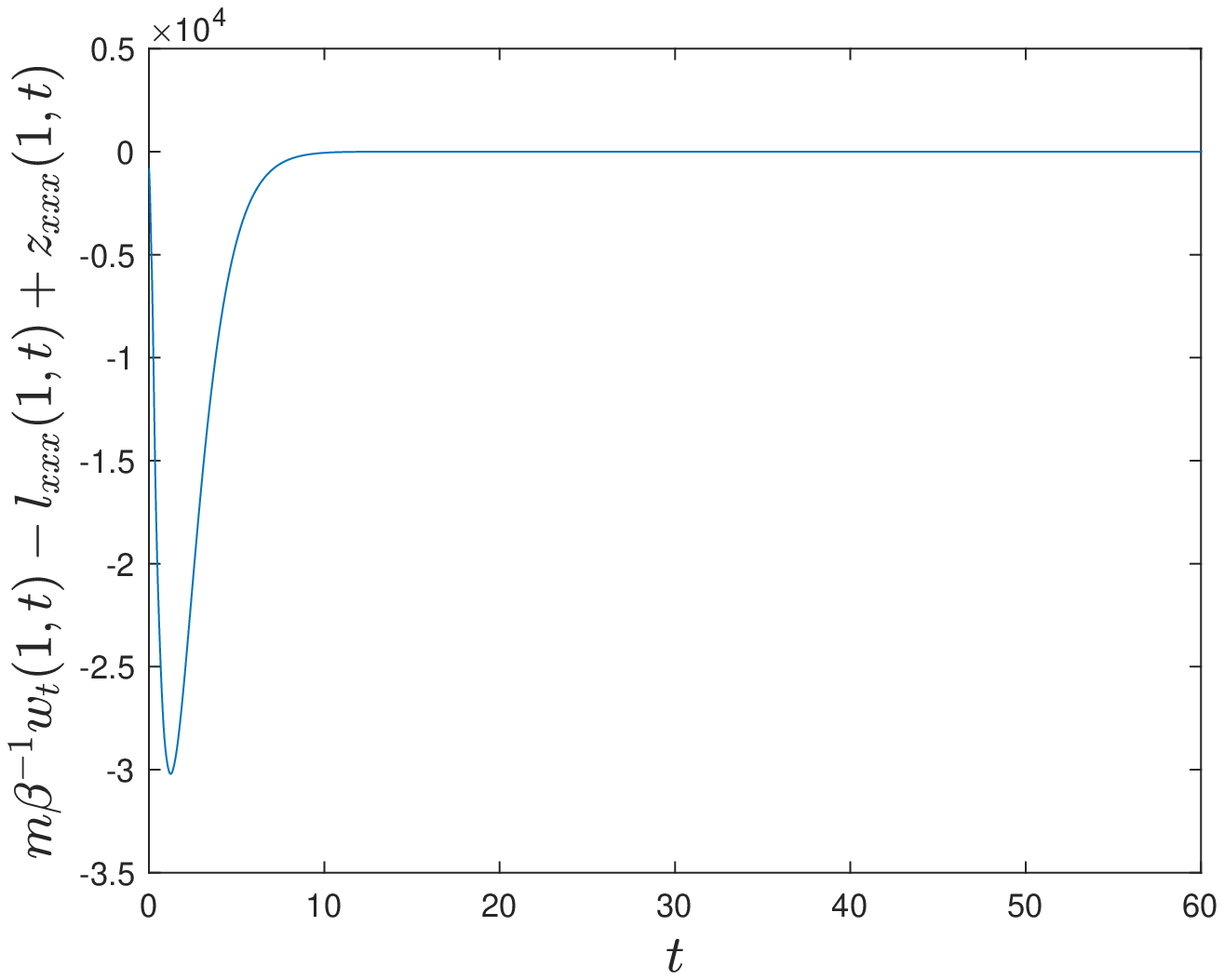}
\caption{The state $\eta(t)=m\beta^{-1}$
$w_t(1,t)-l_{xxx}(1,t)+z_{xxx}(1,t)]$.}
\label{7}
\end{minipage}
\hfill
\begin{minipage}[t]{0.48\linewidth}
\centering
\includegraphics[height=5cm,width=8cm]{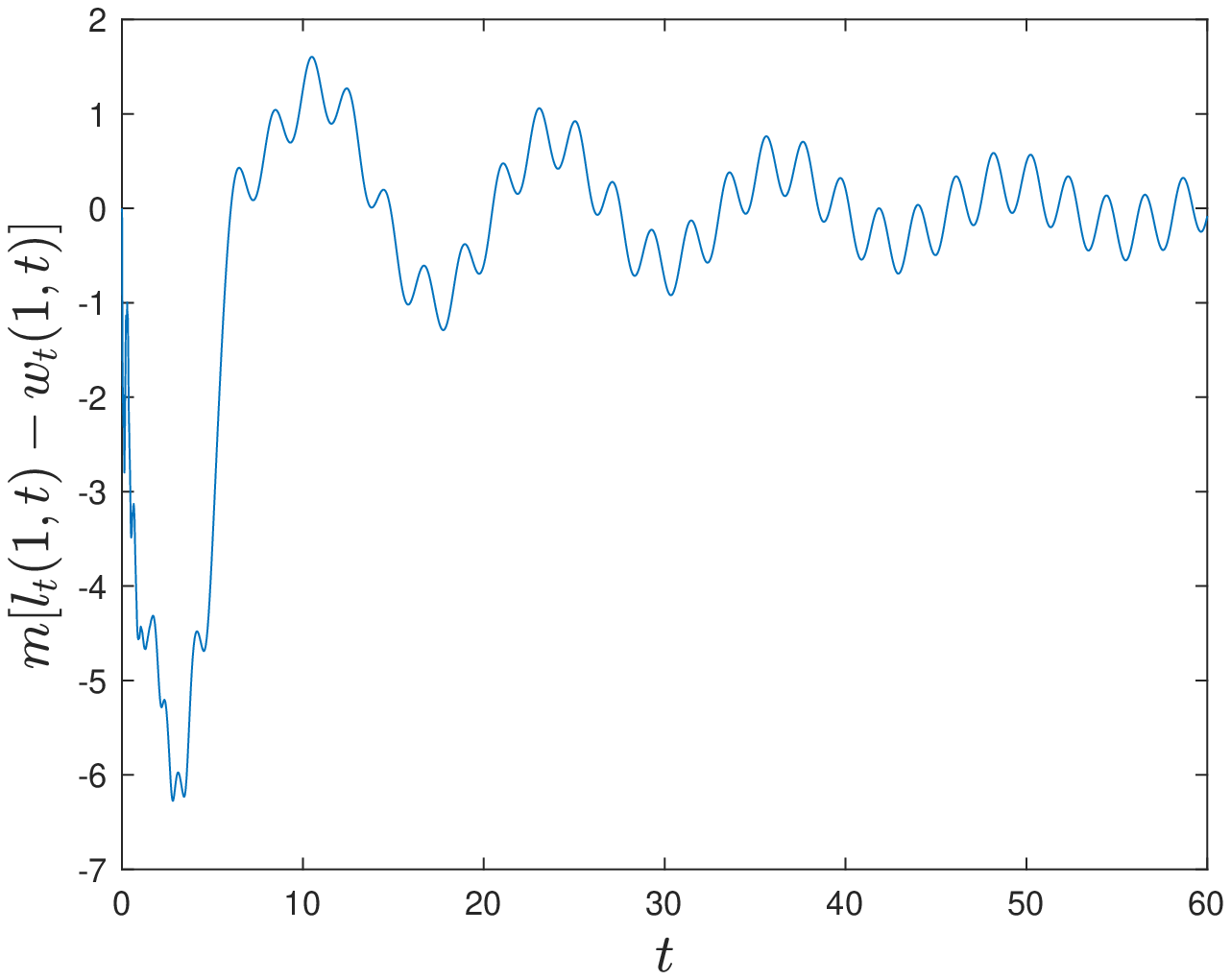}
\caption{The state $\phi(t)=m[l(1,t)-w(1,t)]$.}
\label{8}
\end{minipage}
\end{figure}

\begin{figure}
\begin{minipage}[t]{0.48\linewidth}
\centering     
\includegraphics[height=5cm,width=8cm]{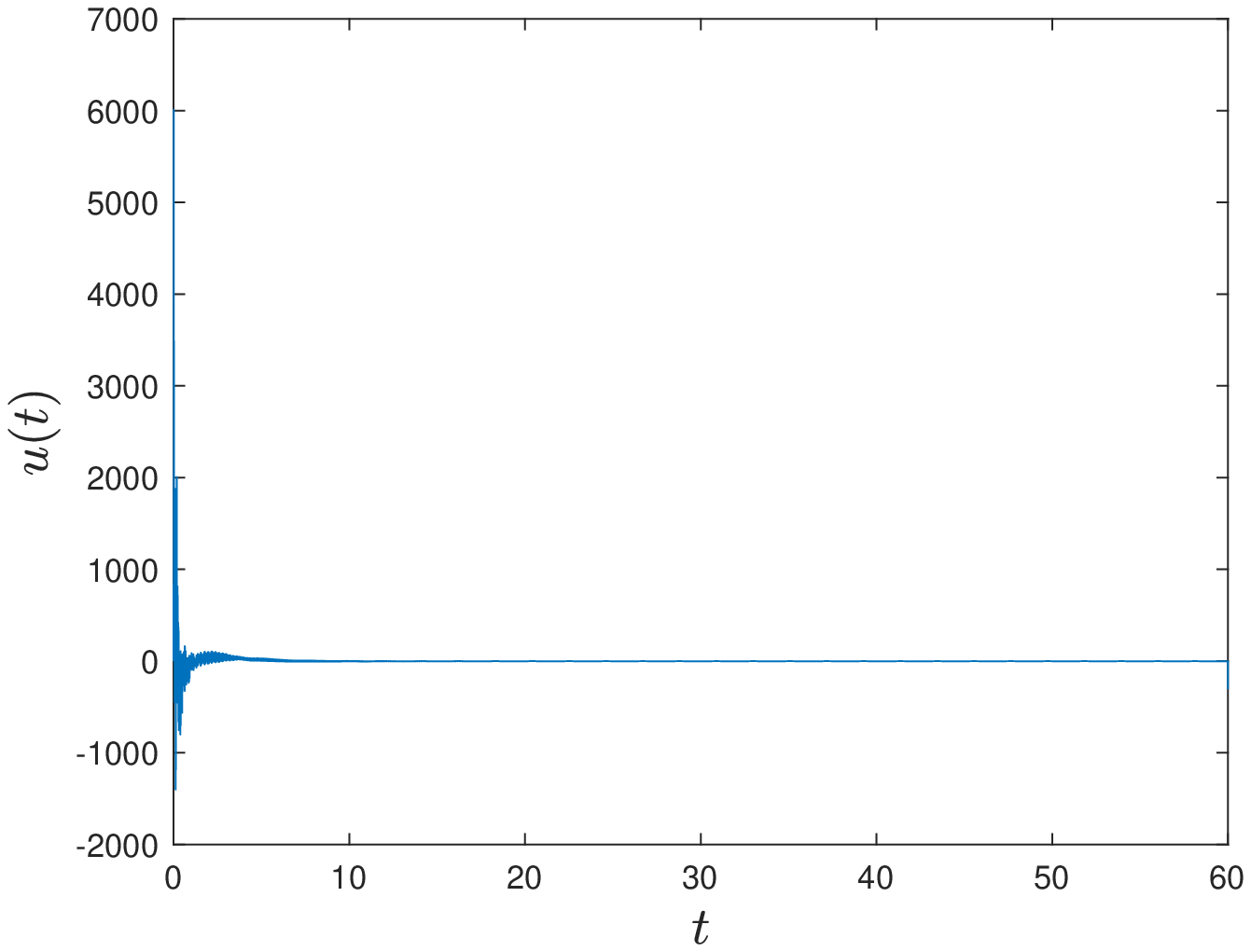}
\caption{The control $u(t)$.}
\label{9}
\end{minipage}
\hfill
\begin{minipage}[t]{0.48\linewidth}
\centering
\includegraphics[height=5cm,width=8cm]{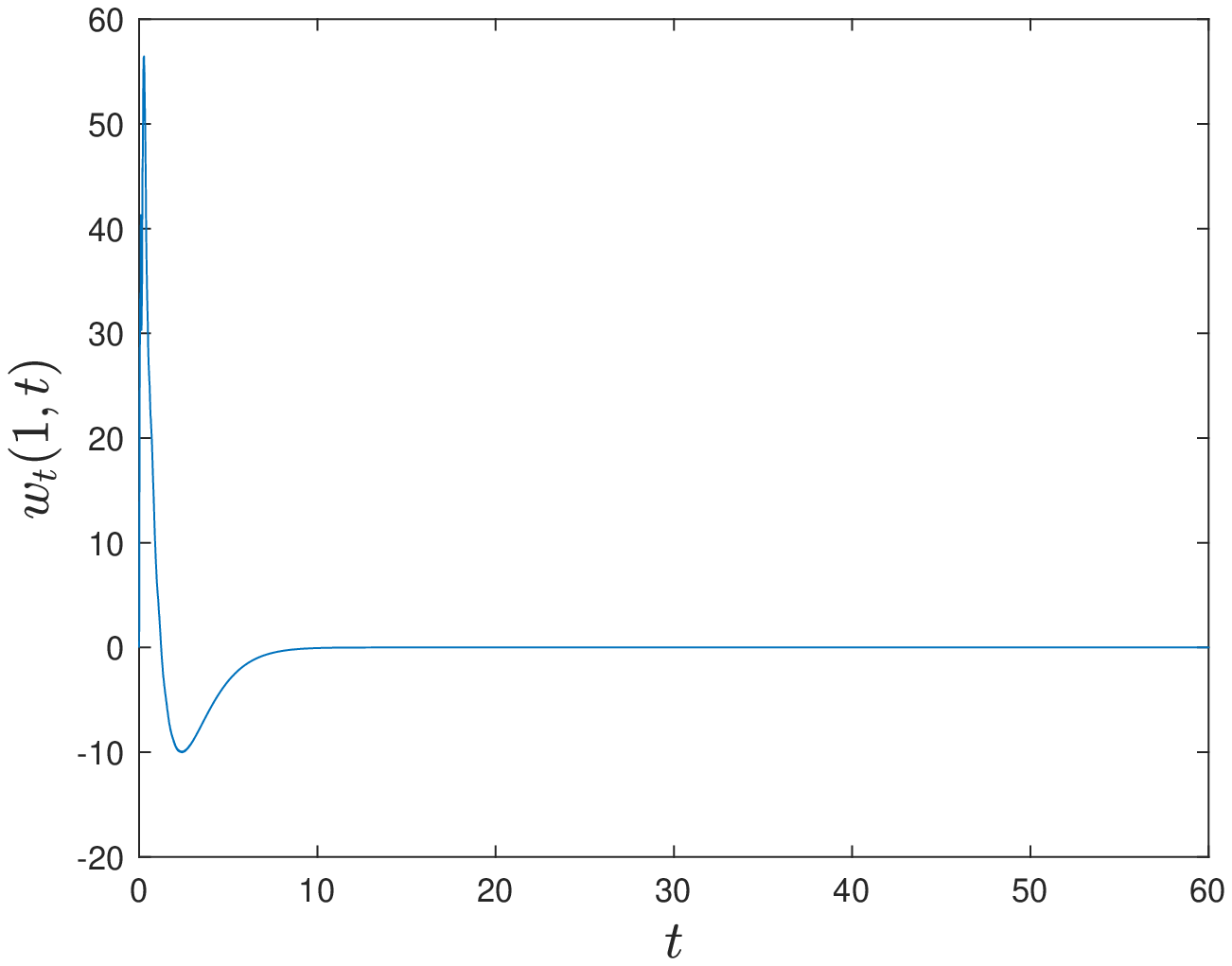}
\caption{$w_t(1,t)$.}
\label{10}
\end{minipage}
\end{figure}

%
%

%
%

\section{Concluding remarks}

In this paper, the dynamic stabilization for a flexible beam with tip mass and with or without disturbance is investigated.
In absence of disturbance, only one non-collocated measurement is used to construct a Luenberger state observer and estimate state based control law to exponentially stabilize the system. This essentially improves the results of the existence references \cite{Conrad1998,Guo2001b}, where two collocated measurements $w_t(1,t)$ and $w_{xxxt}(1,t)$ with $w_{xxxt}(1,t)$ being high order feedback were adopted.
When internal uncertainty and external disturbance are considered,
we construct an infinite-dimensional disturbance estimator to track the state and total disturbance in real time, and design an output feedback control law based on the estimated state and estimated total disturbance, in order
to exponentially stabilize the original system while guaranteeing the boundedness of the closed-loop system.
{\color{blue}Throughout the paper we utilizes many systems which share the same PDE but
different boundary conditions. As a result, when written in abstract setting, these PDEs are
represented by operators with different domains. For example, (\ref{beem}) with $F(t)=0$ and its observer (\ref{transfer11}) do not
have a common domain.}

\section{Appendix}
In this Appendix, we present the output feedback stabilization for system (\ref{beem}) without disturbance ($F(t)=0$), only one measurement $w_{xx}(0,t)$ is available. We design an infinite-dimensional Luenberger state observer described by
\begin{equation} \label{transfer11}
\left\{\begin{array}{l}
\widehat{w}_{tt}(x,t)+\widehat{w}_{xxxx}(x,t)=0, \\
\widehat{w}(0,t)=\widehat{w}_{xx}(1,t)=0, \\
\widehat{w}_{xx}(0,t)=c \widehat{w}_{xt}(0,t)+\gamma \widehat{w}_x(0,t)+w_{xx}(0,t),  \\
-\widehat{w}_{xxx}(1,t)+m\widehat{w}_{tt}(1,t)=u(t),
\end{array}\right.
\end{equation}
where $c,\; \gamma>0$ are tuning parameters.

Set $\widetilde{w}(x,t)=\widehat{w}(x,t)-w(x,t)$ to derive
\begin{equation} \label{wwan}
\left\{\begin{array}{l}
\widetilde{w}_{tt}(x,t)+\widetilde{w}_{xxxx}(x,t)=0,\\
\widetilde{w}(0,t)=\widetilde{w}_{xx}(1,t)=0,\\
\widetilde{w}_{xx}(0,t)=c\widetilde{w}_{xt}(0,t)+\gamma \widetilde{w}_x(0,t), \\
-\widetilde{w}_{xxx}(1,t)+m\widetilde{w}_{tt}(1,t)=0,
\end{array}\right.
\end{equation}
which is written abstractly by
\begin{align}\label{wwansemigroup}
    \frac{d}{dt}(\widetilde{w}(\cdot,t),\widetilde{w}_t(\cdot,t),m\widetilde{w}_t(1,t))=\mathbf{A}_2(\widetilde{w}(\cdot,t),
    \widetilde{w}_t(\cdot,t),m\widetilde{w}_t(1,t)).
\end{align}

\begin{lemma}\label{A2stable}
There exist a sequence of generalized eigenfunctions of operator {\color{blue}$\mathbf{A}_2$} that forms Riesz basis for $\mathbf{H}_2$.
Moreover, {\color{blue}$\mathbf{A}_2$} generates an exponentially stable $C_0$-semigroup on $\mathbf{H}_2$.
\end{lemma}

\begin{proof}\ \
We divide the proof into several steps.

{\bf Step 1.}
We show that operator {\color{blue}$\mathbf{A}_2$} is a densely defined and discrete operator on $\mathbf{H}_2$.
Indeed, obviously it is densely defined. Let $(f,g,\eta)\in \mathbf{H}_2$. We solve the equation $\mathbf{A}_2(\phi,\psi,l)=(f,g,\eta)$ to get
$$\left\{
  \begin{array}{ll}
    \psi(x)=f(x), x\in (0,1),\\
    \phi(x)=\frac{x^3}{6}[\eta+\int_0^1g(\theta)d\theta]-\frac{x^2}{2}[\eta+\int_0^1\theta g(\theta)d\theta]\\
    -\frac{x}{\gamma}[\eta+cf'(0)+\int_0^1\theta g(\theta)d\theta]-\int_0^x\frac{(x-\theta)^3}{6}g(\theta)d\theta,\\
    l=mf(1).
  \end{array}
\right.$$
This implies that $\mathbf{A}_2^{-1}$ is well defined on $\mathbf{H}_2$ and it maps $\mathbf{H}_2$ into $H^4(0,1)\times H^2(0,1)\times
\mathds{C}$. Since by Sobolev imbedding theory $H^4(0,1)\times H^2(0,1)\times \mathds{C}$ is compactly imbedding in $H^2(0,1)\times L^2(0,1)\times \mathds{C}$, $\mathbf{A}_2^{-1}$ is compact and {\color{blue}$\mathbf{A}_2$} a discrete operator.

{\bf Step 2:} We prove that for any $\lambda=i\tau^2\in\sigma(\mathbf{A}_2)$, there corresponds one eigenfunction of the form
$(f,\lambda f,m\lambda f(1))$, where $f$ is given by
 \begin{equation}\label{fxequ300}
  \begin{split}
     &f(x)=(\sinh\tau+\sin\tau)(\cosh\tau x-\cos\tau x)+\bigg[\frac{2\tau\sin\tau}{ic\tau^2+\gamma}\\
     &-(\cosh\tau+\cos\tau)\bigg]\sinh\tau x+\bigg[(\cosh\tau+\cos\tau)+\frac{2\tau\sinh\tau}{ic\tau^2+\gamma}\bigg]\sin\tau x,
  \end{split}
  \end{equation}
and $\tau$ satisfies the characteristic equation
  \begin{align}\label{tauchar200}
  \nonumber&2m\tau^2\sinh\tau\sin\tau+\tau(\cosh\tau\sin\tau-\sinh\tau\cos\tau)\\
&=(ic\tau^2+\gamma)[1+\cosh\tau\cos\tau-m\tau(\cosh\tau\sin\tau-\sinh\tau\cos\tau)].
  \end{align}
This implies that each eigenvalue of $\mathbf{A}_2$ is geometrically simple.

Indeed, $f$ satisfies
\begin{equation}
  \begin{cases}
    f^{(4)}(x)-\tau^4f(x)=0,\\
    f(0)=f''(1)=0,\\
    f^{''}(0)=(ic\tau^2+\gamma)f'(0),\ f^{'''}(1)=-m\tau^4f(1).
  \end{cases}
\end{equation}
The solution of ode $f^{(4)}(x)-\tau^4f(x)=0$ with boundary condition $f(0)=0$ is of the form
\begin{equation}\label{fxc1c2c3c4}
  f(x)=a_1(\cosh\tau x-\cos\tau x)+a_2\sinh\tau x+a_3\sin\tau x.
\end{equation}
By the boundary condition $f''(1)=0,  f^{''}(0)=(ic\tau^2+\gamma)f'(0)$, we can derive (\ref{fxequ300}). The equation (\ref{tauchar200}) is derived by the boundary condition $f^{'''}(1)=-m\tau^4f(1).$

{\bf Step 3:} We show that the eigenvalues $\{\lambda_n,\overline{\lambda_n}\}$, $\lambda_n=i\tau_n^2$ of $\mathbf{A}_2$ have
asymptotic expression
\begin{equation}\label{ataunasy100}
    \tau_n=p\pi+O(n^{-1}),\; \lambda_n=-\frac{2}{c}+i\left[(p\pi)^2+\frac{1}{m}\right]+O(n^{-1}),
\end{equation}
where $n$ is a large positive integer and $p=n+1/4$;
the corresponding eigenfunction $(f_n,\lambda_nf_n,$
$m\lambda_nf(1))$ of $\mathbf{A}_2$ can be chosen by
 \begin{align}\label{fxequasy300}
     F_n(x)=\left(
              \begin{array}{c}
                e^{-p\pi x}+ \cos p\pi x-\sin p\pi x \\
                i[e^{-p\pi x}-\cos p\pi x+\sin p\pi x)] \\
0\\
0\\
              \end{array}
            \right)+O(n^{-1}),
  \end{align}
 where
  \begin{equation}\label{fmatrix100}
    F_n(x)=\frac{2}{\tau_n^2e^{\tau_n}}\begin{pmatrix}
      f_{n}^{''}(x) \\
      \lambda_nf_n(x)\\
      \gamma f'_n(0)\\
      m\lambda f_n(1)\\
    \end{pmatrix}^T,\; \lim_{n\rightarrow\infty}\|F_n\|^2_{\left(L^2(0,1)\right)^2\times \mathds{C}^2}=2.
  \end{equation}

In fact, when ${\rm Im}(\tau)$ is bounded and ${\rm Re}(\tau)\rightarrow\infty$, by (\ref{tauchar200}) we derive the following equality
\begin{align}\label{ruxie0}
       \nonumber &\sin\tau-\cos\tau=\\
       \nonumber&\frac{[\tau\sinh\tau+(ic\tau^2+\gamma)(\cosh\tau+m\tau\sinh\tau)]\cos\tau+ic\tau^2+\gamma}
{2m\tau^2\sinh\tau+\tau\cosh\tau+(ic\tau^2+\gamma)m\tau\cosh\tau}\\
&-\cos\tau=\frac{ic-2m}{icm\tau}\cos\tau+O(|\tau|^{-2}) =O(|\tau|^{-1}).
\end{align}
{\color{blue}Square both sides of the equation (\ref{ruxie0}) to derive $1-2\sin\tau\cos\tau =O(|\tau|^{-2})$}, which implies $\sin(2\tau)=1+O(|\tau|^{-1})$.
{\color{blue}Using Rouche's theorem to derive that equation (\ref{ruxie0}) has solutions satisfy that $\tau_n=(n+\frac{1}{4})\pi+O(n^{-1})=p\pi+O(n^{-1})$, as $n$ large enough.
Combine this,(\ref{ruxie0}), and the Taylor expansions of $\sin\tau$ and $\cos\tau$,} to get $\tau_n=-\frac{2}{c}+i\left[(p\pi)^2+\frac{1}{m}\right]+O(n^{-1})$. Hence (\ref{ataunasy100}) is derived.

{\color{blue}By (\ref{ataunasy100}) and the Taylor expansions of $\sin\tau$ and $\cos\tau$,}  for $y\geq0$ and $0\leq x\leq1$, we have
 \begin{equation}\label{ataunxmx00}
   \begin{cases}
     e^{-\tau_ny}=e^{-p\pi y}+O(n^{-1}),\\
     \sin\tau_nx=\sin p\pi x+O(n^{-1}),\\
     \cos\tau_nx=\cos p\pi x+O(n^{-1}).
   \end{cases}
 \end{equation}
 Obviously, $f_n(x)=f(x)$ is defined by (\ref{fxequ300}) with $\tau=\tau_n$. By (\ref{ataunxmx00}), it follows that
  \begin{align}
   \nonumber &\frac{2}{\tau^2e^\tau}f^{''}(x)=\frac{2}{e^\tau}\bigg\{(\sinh\tau+\sin\tau)(\cosh\tau x+\cos\tau x)+\bigg[\frac{2\tau\sin\tau}{ic\tau^2+\gamma}\\
   \nonumber  &-(\cosh\tau+\cos\tau)]\sinh\tau x-\bigg[(\cosh\tau+\cos\tau)+\frac{2\tau\sinh\tau}{ic\tau^2+\gamma}\bigg]\sin\tau x\bigg\}\\
    &=e^{-p\pi x}+\cos p\pi x-\sin p\pi x+O(n^{-1}),
    \end{align}
  \begin{align}
     \nonumber  &\frac{2}{\tau^2e^\tau}(\lambda f(x))=\frac{2i}{e^\tau}\bigg\{(\sinh\tau+\sin\tau)(\cosh\tau x+\cos\tau x)+\bigg[\frac{2\tau\sin\tau}{ic\tau^2+\gamma}\\
    \nonumber &-(\cosh\tau+\cos\tau)\bigg]\sinh\tau x-\bigg[(\cosh\tau+\cos\tau)+\frac{2\tau\sinh\tau}{ic\tau^2+\gamma}\bigg]\sin\tau x\bigg\}\\
      &=i[e^{-p\pi x}-\cos p\pi x+\sin p\pi x]+O(n^{-1}),
\end{align}
\begin{align}
       \frac{2}{\tau^2e^\tau}\gamma f'(0)=\frac{2}{\tau^2e^\tau}\gamma\frac{2\tau^2(\sinh\tau+\sinh\tau)}{ic\tau^2+\gamma}=O(n^{-1}),
\end{align}
and
\begin{align}
       \nonumber&\frac{2}{\tau^2e^\tau}m\lambda f(1)=\frac{2}{\tau^2e^\tau}\lambda^{-1} f'''(1)\\
      \nonumber &=\frac{2\tau^3}{i\tau^4e^\tau}(\sinh\tau+\sin\tau)(\sinh\tau -\sin\tau)+\frac{2\tau^3}{i\tau^4e^\tau}\bigg[\frac{2\tau\sin\tau}{ic\tau^2+\gamma}\\
    \nonumber &-(\cosh\tau+\cos\tau)\bigg]\cosh\tau-\frac{2\tau^3}{i\tau^4e^\tau}\bigg[(\cosh\tau+\cos\tau)\\
     &+\frac{2\tau\sinh\tau}{ic\tau^2+\gamma}\bigg]\cos\tau=O(n^{-1}).
\end{align}
Accordingly, (\ref{fxequasy300}) and (\ref{fmatrix100}) are obtained.

{\bf Step 4:} We verify that for any $\lambda\in \sigma(\mathbf{A}_2)$, ${\rm Re}\lambda<0$.
To this end, we observe that by step 1, $0\notin \sigma(\mathbf{A}_2)$. Let $0\neq\lambda\in \sigma(\mathbf{A}_2)$ and $(f,\lambda f,m\lambda f(1))$ be the corresponding eigenfunction, we obtain
$\lambda^2f+f^{(4)}=0$ with the boundary conditions $f(0)=f''(1)=0, f''(0)=(c\lambda+\gamma)f'(0),\ f^{'''}(1)=m\lambda^2f(1).$

Multiply the ODE with $\overline{f}$ and integral over $(0,1)$ on both sides to derive
\begin{align}\label{IntODE}
  \lambda^2\int_0^1|f(x)|^2dx+\int_0^1f^{(4)}(x)\overline{f(x)}dx=0.
\end{align}
Use partial integrate and the boundary conditions to get
 \begin{align}\label{IntODEJian}
 \lambda^2\int_0^1|f(x)|^2dx+\int_0^1|f''(x)|^2dx+m\lambda^2|f(1)|^2+(c\lambda+\gamma)|f'(0)|^2=0.
\end{align}

If $\lambda$ is real number, then $\lambda<0$. Indeed, for $\lambda>0$, (\ref{IntODEJian}) indicates $f=0$,
which contradicts with the assumption that $(f,\lambda f,m\lambda f(1))$ is eigenfunction.

If $\lambda$ is not real number, then $\lambda=\lambda_1+i\lambda_2$ with $\lambda_1$ and $\lambda_2\neq 0$ being two real numbers.
We take the imaginary part of (\ref{IntODEJian}) to derive
\begin{align}\label{Imag}
  \lambda_2\bigg[2\lambda_1\int_0^1|f(x)|^2dx+2m\lambda_1|f(1)|^2+c|f'(0)|^2\bigg]=0.
\end{align}
Since $\lambda_2\neq 0$, (\ref{Imag}) indicates $\lambda_1\leq 0$. For $\lambda_1=0$, we derive $f'(0)=0$. {\color{blue}Then
 we use the boundary condition to derive $f''(0)=(c\lambda+\gamma)f'(0)=0$.} Denote $\lambda_2=\tau^2$ or
$\lambda_2=-\tau^2$, where $\tau>0$. The ODE becomes $f^{(4)}-\tau^4 f=0$ with boundary conditions $f(0)=f'(0)=f''(0)=f''(1)=0$ and $f'''(1)=-m\tau^4 f(1)$.
The boundary condition $f(0)=f'(0)=f''(0)=0$ indicates that the solution of ODE is $f(x)=B\big(\sinh(\tau x)-\sin(\tau x)\big)$, where $B$ is a constant to be determined. Since $(f,\lambda f,\lambda m f(1))$ is eigenfunction, $B\neq 0$.
Use the boundary conditions $f''(1)=0$ to get $\sinh\tau+\sin\tau=0$.
Observe that $g(t)=\sinh(t)+\sin(t)$ is a strictly monotone increasing function on $(-\infty, \infty)$ and $g(0)=0$.
Then $\tau=0$. This contradicts with $\lambda_2\neq 0$. Therefore $\lambda_2<0$.

{\bf Step 5:} By simple calculation, we obtain that the operator $J_0: D(J_0)(\subset\mathbf{H}_2)\rightarrow \mathbf{H}_2$
{\color{blue}\begin{align*}
    \left\{
      \begin{array}{ll}
        J_0(f,g,\eta)=(g(x),-f^{(4)},f'''(1)), \;\forall (f,g,\eta)\in D(J_0),\\
        D(J_0)=\{(f,g,\eta)\in(H^4(0,1)\cap H_L^2(0,1))\times H_L^2(0,1)\\
\times\mathds{R}|f''(0)=g'(0)=0,\eta=mg(1)\}
      \end{array}
    \right.
\end{align*}
}
of the system
\begin{equation}\label{fangonge}
  \begin{cases}
  w_{tt}(x,t)+w_{xxxx}(x,t)=0,\ 0<x<1,\\
  w(0,t)=w_{xx}(1,t)=w_{xt}(0)=0,\\
  -w_{xxx}(1,t)+mw_{tt}(1,t)=0.
\end{cases}
\end{equation}
is skew-adjoint and have compact resolvent. Moreover, we can obtain that $1\in\rho(J_0)$ and $(I-J_0)^{-1}$ is compact on
$\mathbf{H}_2$. Hence $J_0$ a skew-adjoint and discrete operator, and it has a sequence of normalized eigenfunctions
that forms an orthonormal basis for $\mathbf{H}_2$. By putting $\gamma=c^{-1}=0$ in Step 3, one derives
that all the eigenvalues  $\{\mu_n,\overline{\mu_n}\}$, $\mu_n=i\omega_n^2$ are algebraically simple and have asymptotic expression
\begin{equation*}
  \omega_n=p\pi+O(e^{-n}), p=n+1/4;
\end{equation*}
and the eigenfunction have asymptotic expression
\begin{align}\label{aphixasy00}
   G_n(x)=\left(
              \begin{array}{c}
                e^{-p\pi x}+\cos p\pi x-\sin p\pi x \\
                i[e^{-p\pi x}-\cos p\pi x+\sin p\pi x] \\
                0\\
                0\\
              \end{array}
            \right)+O(n^{-1}),
  \end{align}
 where
  \begin{equation}
G_n(x)=\frac{2}{\omega_n^2e^{\omega_n}}\begin{pmatrix}
      \phi_{n}''(x) \\
      \mu_n\phi_n(x)\\
      0\\
      m\mu_n\phi_n(1)\\
    \end{pmatrix}^\top.
  \end{equation}
Here $(\phi_n,\mu_n\phi_n,0)$ is the eigenfunction of $J_0$ corresponding to the eigenvalue $\mu_n$.

{\bf Step 6:} We prove that there exist a sequence of generalized eigenfunctions of $\mathbf{A}_2$ which forms Riesz basis for $\mathbf{H}_2$.
To this end, {\color{blue}similar to Guo, Wang and Yang \cite{Guo2008}, we define an isometric isomorphism $\mathbf{T}:\mathbf{H}_2\rightarrow L^2(0,1)\times L^2(0,1)\times \mathds{R}^2$
by $\mathbf{T}(f,g,h)=(f'',g,\gamma f'(0),h),(f,g,h)\in \mathbf{H}_2$.}
Then, by (\ref{fxequasy300}) and (\ref{aphixasy00}), there exists $N>0$ such that
\begin{equation}
\begin{split}
  &\sum_{n>N}^{\infty}\bigg\|\frac{2}{\tau_n^2e^{\tau_n}}(f_n,\lambda_nf_n,m\lambda f_n(1))
-\frac{2}{\omega_n^2e^{\omega_n}}(\phi_n,\mu_n\phi_n,m\mu \phi_n(1))\bigg\|^2_{\mathbf{H}_2}\\
&{\color{blue}=\sum_{n>N}^{\infty}\bigg\|\frac{2}{\tau_n^2e^{\tau_n}}\mathbf{T}(f_n,\lambda_nf_n,m\lambda f_n(1))
-\frac{2}{\omega_n^2e^{\omega_n}}\mathbf{T}(\phi_n,\mu_n\phi_n,m\mu \phi_n(1))\bigg\|^2_{(L^2(0,1))^2\times \mathbb{C}^2}}\\
  &=\sum_{n>N}^{\infty}\|F_n-G_n\|^2_{(L^2(0,1))^2\times \mathbb{C}^2}=\sum_{n>N}^{\infty}O(n^{-2})<\infty.
\end{split}
\end{equation}
The same thing is true for conjugates. By step 1, $\mathbf{A}_2$ is a densely defined and discrete operator. Therefore $\mathbf{A}_2$ satisfies all the conditions of Theorem 1 of \cite{Guo2001}. Hence the generalized eigenfunctions of $\mathbf{A}_2$ forms Riesz basis for $\mathbf{H}_2$.

{\bf Step 7:} We show that $\mathbf{A}_2$ generates an exponentially stable $C_0$-semigroup on $\mathbf{H}_2$.
Indeed, the Riesz basis generation in step 6 indicates that $\mathbf{A}_2$ generates a $C_0$-semigroup and
the spectrum determine growth condition holds. Then the exponential stability is obtained from step 4 and
the asymptotic expression (\ref{ataunasy100}).
This completes the proof.
\end{proof}

We design an estimated state based output feedback control law
\begin{align}\label{control}
 u(t)=-\alpha\widehat{w}_t(1,t)+\beta \widehat{w}_{xxxt}(1,t)
\end{align}
 to derive the closed-loop system
\begin{align}\label{closednodisturbance}
    \left\{\begin{array}{l}
w_{tt}(x,t)+w_{xxxx}(x,t)=0, \\
w(0,t)=w_x(0,t)=w_{xx}(1,t)=0,\\
-w_{xxx}(1,t)+mw_{tt}(1,t)=-\alpha\widehat{w}_t
+\beta \widehat{w}_{xxxt}(1,t),\\
\widehat{w}_{tt}(x,t)+\widehat{w}_{xxxx}(x,t)=0, \\
\widehat{w}(0,t)=\widehat{w}_{xx}(1,t)=0,\\
\widehat{w}_{xx}(0,t)=c \widehat{w}_{xt}(0,t)+\gamma \widehat{w}_x(0,t)
+w_{xx}(0,t),  \\
-\widehat{w}_{xxx}(1,t)+m\widehat{w}_{tt}(1,t)=-\alpha\widehat{w}_t
+\beta \widehat{w}_{xxxt}(1,t),
\end{array}\right.
\end{align}
which is equivalent to
\begin{align}\label{closednodisturbance1}
    \left\{\begin{array}{l}
w_{tt}(x,t)+w_{xxxx}(x,t)=0,\\
w(0,t)=w_x(0,t)=w_{xx}(1,t)=0, \\
-w_{xxx}(1,t)+mw_{tt}(1,t)=-\alpha w_t(1,t)+\beta w_{xxxt}(1,t)\\
-\alpha\widetilde{w}_t+\beta \widetilde{w}_{xxxt}(1,t), \\
\widetilde{w}_{tt}(x,t)+\widetilde{w}_{xxxx}(x,t)=0,\\
\widetilde{w}(0,t)=\widetilde{w}_{xx}(1,t)=0,\\
\widetilde{w}_{xx}(0,t)=c\widetilde{w}_{xt}(0,t)+\gamma \widetilde{w}_x(0,t), \\
-\widetilde{w}_{xxx}(1,t)+m\widetilde{w}_{tt}=0.
\end{array}\right.
\end{align}
Consider system (\ref{closednodisturbance1}) in Hilbert state space $\mathcal{H}=\mathbf{H}\times \mathbf{H}_2$, where $\mathbf{H}=H_E^2(0,1)\times L^2(0,1)\times\mathds{C}$. The norm is given by
$\|(f,g,\eta)\|^2_{\mathbf{H}} = \int_0^1[|f''(x)|^2+|g(x)|^2]dx+\frac{\beta^2}{m+\alpha\beta} |\eta|^2,$
$(f,g)\in \mathbf{H}.$
Define the operator $\mathbf{A}:D(\mathbf{A})(\subset \mathbf{H})\rightarrow \mathbf{H}$ by
$\mathbf{A}(f,g,\eta)=(g,-f^{(4)},-\eta\beta^{-1}-\beta^{-1}(\alpha
    -m\beta^{-1})g(1)), $
    $   \forall\;{\color{blue}(f,g,\eta)}\in D(\mathbf{A}),$
$\color{blue} D(\mathbf{A})=\{(f,g)\in (H^4(0,1)\bigcap H_E^2(0,1))\times H_E^2(0,1)| f''(1)=0,\eta=-f'''(1)
   +m\beta^{-1}g(1)\}.$
By \cite{Conrad1998,Guo2001b}, it follows that $\mathbf{A}$ generates an exponentially stable $C_0$-semigroup, and there exist a sequence of generalized eigenfunctions of $\mathbf{A}$ that forms Riesz basis for $\mathbf{H}$.
Furthermore, we define the operator $\mathcal{A}:D(\mathcal{A})(\subset \mathcal{H})\rightarrow \mathcal{H}$ by
$\mathcal{A}(f,g,\eta,\phi,\psi,h)=(g,-f^{(4)},-\eta\beta^{-1}-\beta^{-1}(\alpha
    -m\beta^{-1})g(1)-\beta^{-1}\phi'''(1)
-\alpha\beta^{-1}\psi(1),\psi,
-\phi^{(4)},\phi'''(1)),\; (f,g,{\color{blue}\eta},\phi,\psi,h)\in D(\mathcal{A}),$
    $
    D(\mathcal{A})=\{(f,g,\eta,\phi,\psi,h)\in (H_E^2(0,1)\bigcap H^4(0,1))
    \times H_E^2(0,1)\times D(\mathbf{A}_2),
    f''(1)=0,\eta=-f'''(1)
    +m\beta^{-1}g(1)-\phi'''(1), \},$
where $\mathcal{H}=\mathbf{H}\times \mathbf{H}_2$.

System (\ref{closednodisturbance1}) is abstractly described by
$\frac{d}{dt}Z(t)=\mathcal{A}Z(t),$
where $Z(t)=\big(w(\cdot,t),w_t(\cdot,t),$
$-w_{xxx}(1,t)-\widetilde{w}_{xxx}(1,t)$
$+m\beta^{-1}w_t(1,t),\widetilde{w}(\cdot,t),\widetilde{w}_t(\cdot,t),m\widetilde{w}(1,t)\big).$
{\color{blue}In the current inner product one cannot directly show that the operator $\mathcal{A}$ is dissipative.}
The multiplier method is not effective to verify the stability of the semigroup $e^{\mathcal{A}t}$;
it seems difficult to find an equivalent inner product in $\mathcal{H}$ to make $\mathcal{A}$ dissipative.
Instead, we shall use Riesz basis approach to verify the stability. However, because of the complexity of beam equation, it
is not so easy to prove the Riesz basis generation of couple beam equations \cite{Guo2008,Guo2019}.
In the following theorem, we shall apply Bari's theorem to the verification of Riesz basis generation by finding out complicated
relations between sequences of generalized eigenfunctions.

\begin{theorem}\label{exponentialnodisturbance}
The operator $\mathcal{A}$ generates an exponentially stable $C_0$-semigroup on $\mathcal{H}$:
there exist two positive constants $M_\mathcal{A}$ and $\omega_\mathcal{A}$ such that
\begin{align*}
  \|e^{\mathcal{A}t}\|_\mathcal{H}\leq M_\mathcal{A}e^{-\omega_\mathcal{A} t},\;\; t\geq 0.
\end{align*} Moreover,
for any initial condition
$\big(w(\cdot,0),w_t(\cdot,0),-\widehat{w}_{xxx}(1,0)+m\beta^{-1}w_t(1,0),\widehat{w}(\cdot,0),\widehat{w}_t(\cdot,0),$
$m(\widehat{w}_t(1,0)-w_t(1,0))\big)\in \mathcal{H}$, the state $\big(w(\cdot,t),w_t(\cdot,t),-\widehat{w}_{xxx}(1,t)+m\beta^{-1}w_t(1,t),\widehat{w}(\cdot,t),\widehat{w}_t(\cdot,t),$
$m(\widehat{w}_t(1,t)-w_t(1,t))\big)\in C((0,\infty),\mathcal{H}) $ of
(\ref{closednodisturbance1}) is exponentially stable in the sense that
$\|\big(w(\cdot,t),$
$w_t(\cdot,t),-\widehat{w}_{xxx}(1,t)+m\beta^{-1}w_t(1,t),\widehat{w}(\cdot,t),
\widehat{w}_t(\cdot,t),m(\widehat{w}_t(1,t)-w_t(1,t))\big)\|_{\mathcal{H}}\leq M_1e^{-{\color{blue}\omega_\mathcal{A}} t},$
where $M_1$ is a positive constant.
\end{theorem}
\begin{proof}\ \
It is routine to show that $\mathcal{A}^{-1}$ exists and is compact on $\mathcal{H}$, that is,
 $\mathcal{A}$ is a discrete operator and the spectrum consists of eigenvalues.

Now we shall show that $\sigma(\mathcal{A})=\sigma(\mathbf{A})\bigcup \sigma(\mathbf{A}_2)$.
Observe that by \cite{Guo2008} and Lemma \ref{Riesz}, $\mathbf{A}$ and $\mathbf{A}_2$ are also discrete operators, thereby, it is easily seen that $\sigma(\mathcal{A})\supseteq\sigma(\mathbf{A})\bigcup \sigma(\mathbf{A}_2).$
Let $\lambda\in \sigma(\mathcal{A})$ and $(f,g,\eta,\phi,\psi,h)$ be the corresponding eigenfunction.
If $(\phi,\psi,h)\neq 0$, $\lambda\in\sigma(\mathbf{A}_2)$; if $(\phi,\psi,h)=0$, we have that $(f,g,\eta)\neq 0$,
$(f,g,\eta)\in D(\mathcal{A})$, and $\lambda (f,g,\eta)=\mathcal{A}(f,g,\eta)$, which implies that $\lambda\in \sigma(\mathcal{A})$. Hence
{\color{blue}$\sigma(\mathcal{A})\subseteq\sigma(\mathbf{A})\bigcup \sigma(\mathbf{A}_2)$}. Therefore $\sigma(\mathcal{A})=\sigma(\mathbf{A})\bigcup \sigma(\mathbf{A}_2)$.

Next, we shall show that the generalized eigenfunction of $\mathcal{A}$ forms a Riesz basis for $\mathcal{H}$.
Let $\{\mu_{n},\overline{\mu}_{n}\}_{n=1}^\infty$ and $\{\lambda_n,\overline{\lambda}_{n}\}_{n=1}^\infty$ be respectively the eigenvalues of $\mathbf{A}$ and $\mathbf{A}_2$, where $\mu_{n}=i\omega_{n}^2,\lambda_{n}=i\tau_{n}^2$.
By \cite[(5)]{Guo2001b} and Lemma \ref{A2stable}, it follows that, the generalized eigenfunctions corresponding to $\{\mu_{n}\}_{n=1}^\infty$ and $\{\lambda_{n}\}_{n=1}^\infty$ can be respectively given by $\{(2\omega_{n}^{-2}e^{-\omega_{n}}f_{n},2ie^{-\omega_{n}}f_{n},-\frac{2i\beta^{-1}}{\mu_n+\beta^{-1}}e^{-\omega_{n}}(\alpha$
$-m\beta^{-1})f_{n}(1))\}_{n=1}^\infty$ and $\{(2\tau_{n}^{-2}e^{-\tau_{n}}\phi_{n},2ie^{-\tau_{n}}\phi_{n},
-2i\tau_{n}^{-4}$
$e^{-\tau_{n}}\phi'''_{n}(1)\}_{n=1}^\infty$ such that the sequences
 $\{(2\omega_{n}^{-2}e^{-\omega_{n}}f_{n},2ie^{-\omega_{n}}f_{n},-\frac{2i\beta^{-1}}{\mu_n+\beta^{-1}}e^{-\omega_{n}}(\alpha-m\beta^{-1})f_{n}(1))\}_{n=1}^\infty\bigcup \{(\overline{2\omega_{n}^{-2}e^{-\omega_{n}}f_{n}},\overline{2ie^{-\omega_{n}}f_{n}},-\overline{\frac{2i\beta^{-1}}{\mu_n+\beta^{-1}}}$
 $\overline{e^{-\omega_{n}}(\alpha
-m\beta^{-1})f_{n}(1)})\}_{n=1}^\infty $ and
$\{(2\tau_{n}^{-2}e^{-\tau_{n}}\phi_{n},2ie^{-\tau_{n}}\phi_{n},
-2i\tau_{n}^{-4}e^{-\tau_{n}}\phi'''_{n}(1)$
$)\}_{n=1}^\infty\bigcup
\{(\overline{2\tau_{n}^{-2}e^{-\tau_{n}}\phi_{n}},$
$\overline{2ie^{-\tau_{n}}\phi_{n}},
\overline{-2i\tau_{n}^{-4}e^{-\tau_{n}}\phi'''_{n}(1))}\}_{n=1}^\infty$
form Riesz basises for $\mathbf{H}$ and $ \mathbf{H}_2$, respectively.
As a result, $\{(2\omega_{n}^{-2}e^{-\omega_{n}}f_{n},
2ie^{-\omega_{n}}f_{n},-\frac{2i\beta^{-1}}{\mu_n+\beta^{-1}}e^{-\omega_{n}}(\alpha-m\beta^{-1})f_{n}(1),0,0,0)\}_{n=1}^\infty
\bigcup \{(0,0,0,2\tau_{n}^{-2}e^{-\tau_{n}}\phi_{n},2ie^{-\tau_{n}}$
$\phi_{n},-2i\tau_{n}^{-4}e^{-\tau_{n}}\phi'''_{n}(1))\}_{n=1}^\infty\bigcup
\{(\overline{2\omega_{n}^{-2}e^{-\omega_{n}}f_{n}},\overline{2ie^{-\omega_{n}}f_{n}},
-\overline{\frac{2i\beta^{-1}}{\mu_n+\beta^{-1}}e^{-\omega_{n}}}$
$\overline{(\alpha-m\beta^{-1})f_{n}(1)},0,0,0)\}_{n=1}^\infty$
$\bigcup\{(0,0,0,
\overline{2\tau_{n}^{-2}e^{-\tau_{n}}\phi_{n}},\overline{2ie^{-\tau_{n}}}$
$\overline{\phi_{n}},\overline{-2i\tau_{n}^{-4}e^{-\tau_{n}}\phi'''_{n}(1)})\}_{n=1}^\infty$
forms a Riesz basis for $\mathbf{H}\times \mathbf{H}_2$, which is equivalent to that $\{(2\omega_{n}^{-2}e^{-\omega_{n}}f''_{n},2ie^{-\omega_{n}}f_{n},-\frac{2i\beta^{-1}}{\mu_n+\beta^{-1}}e^{-\omega_{n}}(\alpha-m\beta^{-1})f_{n}(1),0,
0,0,0)\}_{n=1}^\infty
\bigcup \{(0,$
$0,0,2\tau_{n}^{-2}e^{-\tau_{n}}\phi''_{n},2ie^{-\tau_{n}}\phi_{n},2\tau_{n}^{-2}e^{-\tau_{n}}$
$\gamma\phi_n'(0),
-2i\tau_{n}^{-4}e^{-\tau_{n}}\phi'''_{n}(1))\}_{n=1}^\infty\bigcup
\{(\overline{2\omega_{n}^{-2}e^{-\omega_{n}}f''_{n}},\overline{2ie^{-\omega_{n}}f_{n}},$\\
$-\overline{\frac{2i\beta^{-1}}{\mu_n+\beta^{-1}}e^{-\omega_{n}}}(\alpha-m\beta^{-1})\overline{f_{n}(1)},0,0,0,0)\}_{n=1}^\infty \bigcup \{(0,0,0,$
$\overline{2\tau_{n}^{-2}e^{-\tau_{n}}\phi''_{n}},-2i\overline{e^{-\tau_{n}}\phi_{n}},\overline{2\tau_{n}^{-2}e^{-\tau_{n}}\gamma\phi'_{n}(0)},$
$\overline{-2i\tau_{n}^{-4}e^{-\tau_{n}}\phi'''_{n}(1)})\}_{n=1}^\infty$
forms a Riesz basis for $\left(L^2(0,1)\right)^2\times \mathds{C}^1\times \left(L^2(0,1)\right)^2\times \mathds{C}^2$.

Let $\lambda=i\tau^2\in \sigma(\mathcal{A})$ and $\big(2\tau^{-2}e^{-\tau}f,2ie^{-\tau}f,-\frac{2i\beta^{-1}}{\lambda+\beta^{-1}}e^{-\tau}(\alpha-m\beta^{-1})f(1),
2\tau^{-2}e^{-\tau}\phi,2ie^{-\tau}\phi,$
$-2i\tau^{-4}e^{-\tau}\phi'''(1)\big)$ be the corresponding eigenfunction. If $\phi=0$, then $\big(2\tau^{-2}e^{-\tau}f,2ie^{-\tau}f,
-\frac{2i\beta^{-1}}{\lambda+\beta^{-1}}$
$e^{-\tau}(\alpha-m\beta^{-1})f(1)\big)\neq 0$ and $\lambda\in \sigma(\mathbf{A})$.
Hence in this case the eigenvalues $\{\mu_{n}\}_{n=1}^\infty$ corresponds the eigenfunctions $\{\big(2\omega_n^{-2}e^{-\omega_n}f,2ie^{-\omega_n}f,
-\frac{2i\beta^{-1}}{\mu_n+\beta^{-1}}e^{-\omega_n}(\alpha-m\beta^{-1})f(1),0,0,0\big)\}_{n=1}^\infty)\}_{n=1}^\infty$ of $\mathcal{A}$.

If $\phi\neq 0$, then $\lambda\in \sigma(\mathbf{A}_2)$. The eigenvalues $\{\lambda_{n}\}_{n=1}^\infty$ corresponds
the eigenfunction $\big(2\tau_n^{-2}e^{-\tau_n}\phi_n,$
$2ie^{-\tau_n}\phi_n,-2i\tau_n^{-4}e^{-\tau_n}\phi'''_n(1)\big)$ of $\mathbf{A}_2$, where
\begin{equation}\label{f1n}
  \begin{split}
     &\phi_{n}(x)=(\sinh\tau_n+\sin\tau_n)(\cosh\tau_n x-\cos\tau_n x)\\
     &+\bigg[\frac{2\tau_n\sin\tau_n}{ic\tau_n^2+\gamma}-(\cosh\tau_n+\cos\tau_n)\bigg]\sinh\tau_n x\\
&+\bigg[(\cosh\tau_n+\cos\tau_n)+\frac{2\tau_n\sinh\tau_n}{ic\tau_n^2+\gamma}\bigg]\sin\tau_n x.
  \end{split}
  \end{equation}
Then $i\tau_n^2[\alpha \phi_n(1)-\beta\phi_n'''(1)]=2i\tau_n^2K_n,$ where
\begin{align}
   \nonumber &K_n=[\alpha(\cosh\tau_n\sin\tau_n-\sinh\tau_n\cos\tau_n)+\beta\tau_n^3(1\\
   \nonumber &+\cosh\tau_n\cos\tau_n)]+\frac{\tau_n}{ic\tau_n^2+\gamma}[2\alpha\sinh\tau_n\sin\tau_n\\
&-\beta\tau_n^3(\cosh\tau_n\sin\tau_n-\sinh\tau_n\cos\tau_n)].
\end{align}
Denote by $\big(2\tau_n^{-2}e^{-\tau_n}f_{1n},2ie^{-\tau_n}f_{1n},
-\frac{2i\beta^{-1}}{\lambda_n+\beta^{-1}}e^{-\tau_n}(\alpha-m\beta^{-1})f_{1n}(1),
2\tau_n^{-2}e^{-\tau_n}\phi_n,2ie^{-\tau_n}\phi_n,-2i\tau_n^{-4}$
$e^{-\tau_n}\phi'''_n(1)\big)$ the eigenfunction of $\mathcal{A}$ corresponding to $\lambda_{n}$. Then we have
\begin{align}\label{F11}
    \left\{
      \begin{array}{ll}
        f^{(4)}_{1n}(x)-\tau^4_{n}f_{1n}(x)=0, \\
        f_{1n}(0)=f_n'(0)=f''_{1n}(1)=0,  \\
       (1+i\beta\tau^2)f'''_{1n}(1)=(i\alpha\tau_n^2-m\tau_n^4)f_{1n}(1)+2i\tau_n^2 K_n.
      \end{array}
    \right.
\end{align}
The solution the ode $f^{(4)}_{1n}(x)-\tau^4_{n}f_{1n}(x)=0$ with boundary conditions $f_{n}(0)=f_n'(0)=0$ is of the form
\begin{align}
    f_{1n}(x)=b_1[\cosh\tau_{n}x-\cos\tau_{n}x]+b_2(\sinh\tau_{n}x-\sin\tau_{n}x).
\end{align}
Use the other boundary conditions $f''_{1n}(1)=0, (1+i\beta\tau^2)f'''_{1n}(1)=(i\alpha\tau_n^2-m\tau_n^4)f_{1n}(1)+2i\tau_n^2 K_n$ to get
\begin{align}\label{F12}
    \left\{
      \begin{array}{ll}
        b_1=-\frac{iK_n(\sinh\tau_n+\sin\tau_n)}{h_\tau},\\
        b_2=\frac{iK_n(\cosh\tau_n+\cos\tau_n)}{h_\tau},\\
      \end{array}
    \right.
\end{align}
where $h_\tau=\tau_n(1+\cosh\tau_n\cos\tau_n)(1+i\beta\tau_n^2)-(m\tau_n-i\alpha)(\cosh\tau_n\sin\tau_n-\sinh\tau_n\cos\tau_n).$
This, together with the expression (\ref{ataunxmx00}) to derive
\begin{align}
    \left\{
      \begin{array}{ll}
        2\tau_{n}^{-2}e^{-\tau_{n}}f''_{1n}(x)=-e^{-p\pi x}-\cos p\pi x
        +\sin p\pi x+O(n^{-1})\\
        2\tau_{n}^{-2}e^{-\tau_{n}}[\lambda_nf_{1n}(x)]=i\big[-e^{-p\pi x}+\cos p\pi x
        -\sin p\pi x\big]+O(n^{-1})\\
        -\frac{2i\beta^{-1}}{\lambda_n+\beta^{-1}}e^{-\tau_n}(\alpha-m\beta^{-1})f_{1n}(1)=O(n^{-1}),
      \end{array}
    \right.
\end{align}
where $p=n+1/4$. Denote $Q=\left(
            \begin{array}{cc}
              I_{3\times 3} & J \\
              0_{4\times 3} & I_{4\times 4} \\
            \end{array}
          \right)
$ with $J=\left(
                              \begin{array}{ccccc}
                                -1 & 0& 0 & 0 \\
                                 0 &-1& 0 & 0 \\
                                 0 & 0 & 0 & 0\\
                              \end{array}
                            \right).$
Then $Q$ has bounded inverse and
\begin{align}\label{guanxi1}
\nonumber    &(2\omega_{n}^{-2}e^{-\omega_{n}}f''_{n},2ie^{-\omega_{n}}f_{n},-\frac{2i\beta^{-1}}{\mu_n+\beta^{-1}}e^{-\omega_n}(\alpha
-m\beta^{-1})f_n(1),0,0,0,0)^T\\
&=Q(2\omega_{n}^{-2}e^{-\omega_{n}}f''_{2n},2ie^{-\omega_{n}}f_{n},-\frac{2i\beta^{-1}}{\mu_n+\beta^{-1}}e^{-\omega_n}(\alpha-m\beta^{-1})f_n(1),0,0,0,0)^T,
\end{align}
\begin{align}\label{guanxi2}
\nonumber&\big(2\tau_n^{-2}e^{-\tau_n}f''_{1n},2ie^{-\tau_n}f_{1n},
-\frac{2i\beta^{-1}}{\lambda_n+\beta^{-1}}e^{-\tau_n}(\alpha-m\beta^{-1})f_{1n}(1),
2\tau_n^{-2}e^{-\tau_n}\phi''_n,2ie^{-\tau_n}\phi_n,\\
\nonumber &2\tau_n^{-2}e^{-\tau_n}\gamma\phi'_n(0),-2i\tau_n^{-4}e^{-\tau_n}\phi'''_n(1)\big)^T\\
&=Q(0,0,0,2\tau_{n}^{-2}e^{-\tau_{n}} \phi''_{n},2ie^{-\tau_{n}}\phi_{n},2\tau_n^{-2}e^{-\tau_n}\gamma\phi'_n(0),-2i\tau_n^{-4}e^{-\tau_n}\phi'''_n(1))^T+O(n^{-1}).
\end{align}
The same thing is true for conjugates. Then, by Bari's theorem the sequence $\{(2\omega_{n}^{-2}e^{-\omega_{n}}f''_{n},$
$2ie^{-\omega_{n}}f_{n},-\frac{2i\beta^{-1}}{\mu_n+\beta^{-1}}e^{-\omega_n}(\alpha-m\beta^{-1})f_n(1),0,0,0,0)\}_{n=-\infty}^\infty \bigcup \{(2\tau_{n}^{-2}e^{-\tau_{n}}f''_{1n},2ie^{-\tau_{n}}f_{1n},-\frac{2i\beta^{-1}}{\lambda_n+\beta^{-1}}$\\
$e^{-\tau_n}(\alpha-m\beta^{-1})f_{1n}(1),
2\tau_{n}^{-2}e^{-\tau_{n}} \phi''_{n},2ie^{-\tau_{n}}\phi_{n},2\tau_{n}^{-2}e^{-\tau_{n}} \gamma\phi'_{n}(0),-2i\tau_n^{-4}e^{-\tau_n}\phi'''_n(1))\}_{n=-\infty}^\infty\bigcup
\{(\overline{2\omega_{n}^{-2}}$
$\overline{e^{-\omega_{n}}f''_{n}},$
$\overline{2ie^{-\omega_{n}}f_{n}},-\overline{\frac{2i\beta^{-1}}{\mu_n+\beta^{-1}}e^{-\omega_n}(\alpha-m\beta^{-1})f_n(1)},0,0,0,0)\}_{n=-\infty}^\infty$ $\bigcup\{(\overline{2\tau_{n}^{-2}e^{-\tau_{n}}f''_{1n}},\overline{2ie^{-\tau_{n}}f_{1n}},$\\
$-\overline{\frac{2i\beta^{-1}}{\lambda_n
+\beta^{-1}}e^{-\tau_n}(\alpha-m\beta^{-1})}$
$\overline{f_{1n}(1)},\overline{2\tau_{n}^{-2}e^{-\tau_{n}}\phi''_{n}},\overline{2ie^{-\tau_{n}}\phi_{n}},\overline{2\tau_{n}^{-2}e^{-\tau_{n}}\gamma\phi'_{n}(0)},
\overline{-2i\tau_n^{-4}e^{-\tau_n}}$
$\overline{\phi'''_n(1)})\}_{n=-\infty}^\infty$ forms Riesz basis for $\left(L^2(0,1)\right)^2
\times  \mathds{C}\times \left(L^2(0,1)\right)^2
\times  \mathds{C}^2$,
which is equivalent to that $\{(2\omega_{n}^{-2}e^{-\omega_{n}}f_{n},$\\
$2ie^{-\omega_{n}}f_{n},-\frac{2i\beta^{-1}}{\mu_n+\beta^{-1}}e^{-\omega_n}(\alpha
-m\beta^{-1})f_n(1),$
$0,0,0)\}_{n=-\infty}^\infty \bigcup \{(2\tau_{n}^{-2}e^{-\tau_{n}} f_{1n},2ie^{-\tau_{n}}f_{1n},-\frac{2i\beta^{-1}}{\lambda_n+\beta^{-1}}e^{-\tau_n}$
$(\alpha-m\beta^{-1})f_{1n}(1),
2\tau_{n}^{-2} e^{-\tau_{n}} \phi_{n},2ie^{-\tau_{n}}\phi_{n},-2i\tau_n^{-4}e^{-\tau_n}\phi'''_n(1))$
$\}_{n=-\infty}^\infty\bigcup
\{(\overline{2\omega_{n}^{-2}e^{-\omega_{n}}f_{n}},\overline{2ie^{-\omega_{n}}f_{n}},$
$-\overline{\frac{2i\beta^{-1}}{\mu_n+\beta^{-1}}e^{-\omega_n}(\alpha-m\beta^{-1})f_n(1)},$ $0,0,0)\}_{n=-\infty}^\infty \bigcup \{(\overline{2\tau_{n}^{-2}e^{-\tau_{n}}f_{1n}},\overline{2ie^{-\tau_{n}}f_{1n}},-\overline{\frac{2i\beta^{-1}}{\lambda_n+\beta^{-1}}
e^{-\tau_n}(\alpha-m\beta^{-1})}$
$\overline{f_{1n}(1)},
\overline{2\tau_{n}^{-2}e^{-\tau_{n}}\phi_{n}},\overline{2ie^{-\tau_{n}}\phi_{n}},2i$
$\overline{\tau_n^{-4}e^{-\tau_n}\phi'''_n(1)})\}_{n=-\infty}^\infty$ forms Riesz basis for $\mathbf{H}\times \mathbf{H}_2$.

Finally, we claim that the operator $\mathcal{A}$ generates an exponentially stable $C_0$-semigroup. Indeed the semigroup generation
is directly derived by the Riesz basis property. Moreover, the Riesz basis property implies that the property of spectrum-determined growth condition holds for $\mathcal{A}$. Since $ \mathbf{A}$ and $\mathbf{A}_2$ generate exponentially stable $C_0$-semigroups and
spectrum-determined growth condition holds, $\sup\{{\rm Re} \lambda:\lambda\in \sigma(\mathcal{A})\}=\sup\{{\rm Re} \lambda:\lambda\in \sigma(\mathbf{A})\bigcup \sigma(\mathbf{A}_2)\}<0$. Therefore $e^{\mathcal{A}t}$ is exponentially stable and the proof is
completed.
\end{proof}

In presence of disturbance, we consider the special case $F(t)\equiv F$. With the same design, the closed-loop system has the boundary condition $-\widetilde{w}_{xxx}(1,t)+m\widetilde{w}_{tt}=-F$ and then $w(x,t)=(-x^3/6+x^2/2)F,\widetilde{w}(x,t)=(x^3/6-x^2/2-x/\gamma)F$ is an unstable solution.
This means that when the disturbance is considered, the controller should be redesigned.


\begin{thebibliography}{99}

\bibitem{Balakrishnan1991} A.V. Balakrishnan, Compensator design for stability enhancement with collocated controllers,
{\it IEEE Trans. Automat. Control}, 36 (1991), 994-1008.

\bibitem{Bresch-Pietri2014} D. Bresch-Pietri, M. Krstic, Output-feedback adaptive control of a wave PDE with
boundary anti-damping, {\it Automatica}, 50 (2014), 1407-1415.

\bibitem{Chen1987} G. Chen, M.C. Delfour, A.M. Krall, G. Payre, Modeling stabilization
and control of serially connected beam, {\it SIAM J. Control Optim.}, 25 (1987), 526-546.
%

\bibitem{Cheng2011} M.B. Cheng, V. Radisavljevic, W.C. Su, Sliding mode boundary
control of a parabolic pde system with parameter variations and boundary
uncertainties, {\it Automatica}, 47 (2011), 381-387.

\bibitem{Cannon1984} R.H. Cannon, E. Schmitz, Initial experiments on the end-point control of
a flexible one-link robot, {\it International Journal of Robotics Research}, 3 (1984), 62-75.


\bibitem{Chentouf2006}B. Chentouf, J.M. Wang, Stabilization and optimal decay rate for a non-homogeneous rotating body-beam with dynamic boundary controls,
 {\it Journal of mathematical analysis and applications}, 318(2) (2006), 667-691.

\bibitem{Conrad1998} F. Conrad and O. M\"{o}rg\"{u}l, On the stabilization of a flexible beam with a tip mass, {\it SIAM
J. Control Optim.}, 36 (1998), 1962-1986.

\bibitem{Curtain1986} R.F. Curtain, D. Salamon, Finite dimensional compensators for infinite
dimensional systems with unbounded input operators, {\it SIAM J. Control Optim.}, 24 (1986) 797-816.

\bibitem{Deguenon2006} A.J. Deguenon, G. Sallet, C.Z. Xu, A Kalman observer for
infinite-dimensional skew-symmetric systems with application to an
elastic beam, in Proc. 2nd Int. Symp. Communications, Control, Signal
Processing, Marrakech, Morocco, Mar. 2006.

\bibitem{Feng2017a} H. Feng, B.Z. Guo, A new active disturbance rejection control to output feedback stabilization for a one-dimensional anti-stable wave equation with disturbance, {\it IEEE Trans. Automat. Control}, 62(8) (2017), 3774-3787.


\bibitem{Ge2011a} S.S. Ge, S. Zhang and W. He, Modeling and control of an Euler-Bernoulli beam under unknown spatiotemporally
varying disturbance. IEEE Xplore Conference: American Control Conference, San Francisco, CA, USA, 2011, pp. 2988-2993.

\bibitem{Ge2011b} S.S. Ge, S. Zhang and W. He, Vibration control of an Euler-Bernoulli beam under unknown spatiotemporally
varying disturbance, {\it Int. J. Control}, 84 (2011), 947-960.


\bibitem{Gressang1975} R.V. Gressang, G.B. Lamont, Observers for systems characterized by
semigroups, {\it IEEE Trans. Automat. Control}, 20 (1975), 523-528.

\bibitem{Guillemin1957} E.A. Guillemin, Synthesis of passive networks. New York: Wiley, 1957.

\bibitem{Guo2001} B.Z. Guo, R. Yu, The Riesz basis property of discrete operators and application to a Euler-Bernoulli beam equation with boundary linear feedback control, {\it IMA Journal of Mathematical Control and Information}, 18(2) (2001), 241-251.

\bibitem{Guo2001b} B.Z. Guo, Riesz baisis approach to the stabilization of a flexible beam with a tip mass, {\it SIAM
J. Control Optim.}, 39(6) (2001), 1736-1747.


\bibitem{Guo2008} B.Z. Guo, J.M. Wang, K.Y. Yang, Dynamic stabilization of an Euler-Bernoulli beam under boundary control
and non-collocated observation, {\it Systems \& Control Letters}, 57 (2008), 740-749.


\bibitem{Guo2013a} B.Z. Guo, F.F. Jin, Sliding mode and active disturbance rejection
control to the stabilization of anti-stable one-dimensional wave equation
subject to boundary input disturbance, {\it IEEE Trans. Automat. Control},
58 (2013), 1269-1274.

\bibitem{Guo2013b} B.Z. Guo, F.F. Jin, The active disturbance rejection and sliding mode
control approach to the stabilization of Euler-Bernoulli beam equation
with boundary input disturbance, {\it Automatica}, 49 (2013), 2911-2918.

\bibitem{Guo2014a} B.Z. Guo, W. Kang, Lyapunov approach to the boundary
stabilisation of a beam equation with boundary disturbance, {\it International
Journal of Control}, 87(5) (2014), 925-939.

%

\bibitem{Guo2019} B.Z. Guo, J.M. Wang, Control of wave and beam PDEs: the Riesz basis approach, Springer, 2019.

\bibitem{GuoW2011a} W. Guo, B.Z. Guo, Z.C. Shao, Parameter estimation and stabilization
for a wave equation with boundary output harmonic disturbance and
non-collocated control, {\it International Journal of Robust and Nonlinear
Control}, 21 (2011), 1297-1321.

\bibitem{GuoW2013a} W. Guo, B.Z. Guo, Adaptive output feedback stabilization for one dimensional
wave equation with corrupted observation by harmonic
disturbance, {\it SIAM J. Control Optim.}, 51 (2013), 1679-1706.

\bibitem{GuoW2013b} W. Guo, B.Z. Guo, Parameter estimation and non-collocated adaptive
stabilization for a wave equation subject to general boundary
harmonic disturbance, {\it IEEE Trans. Automat. Control},
58 (2013), 1631-1643.

\bibitem{Han2009} J.Q. Han, From PID to active disturbance rejection control, {\it IEEE
Trans. Ind. Electron.}, 56 (2009), 900-906.

\bibitem{Jin2015} F.F. Jin, B.Z. Guo, Lyapunov approach to output feedback stabilization for Euler-Bernoulli beam equation with boundary input disturbance, {\it Automatica}, 52 (2015), 95-102.

\bibitem{Krstic2010} M. Krstic, Adaptive control of an anti-stable wave PDE, {\it Dyn. Contin. Discrete Impuls. Syst.
Ser. A. Math. Anal.}, 17 (2010),  853-882.


\bibitem{Lasiecka1995} I. Lasiecka, Finite element approximations of compensator design for
analytic generators with fully unbounded controls/observations, {\it SIAM J. Control Optim.}, 33 (1995), 67-88.

\bibitem{Littman1988} W. Littman and L. Markus, Stabilization of a hybrid system of elasticity by feedback boundary
damping, {\it Ann. di Mat. Pura ed Appl.}, 152 (1988), 281-330.

\bibitem{Li2017} Y.F. Li, G.Q. Xu, Z.J. Han, Stabilization of an Euler-Bernoulli beam system with a tip mass
subject to non-uniform bounded disturbance, {\it IMA Journal of Mathematical Control and Information}, 34 (2017),
1239-1254.



\bibitem{Rebarber2003} R. Rebarber, G. Weiss, Internal model based tracking and disturbance rejection for stable well-posed systems,
{\it Automatica}, 39 (2003), 1555-1569.

\bibitem{Rao1995} B.P. Rao, Uniform stabilization of a hybrid system of elasticity, {\it SIAM J. Control Optim.}, 33
(1995), 440-454.




\bibitem{Smyshlyaev2005} A.Smyshlyaev, M. Krstic, Backstepping observers for a class of parabolic PDEs, {\it Syst. Control Lett.}, 54 (2005), 613-625.


\bibitem{Triggiani1989} R. Triggiani, Lack of uniform stabilization for noncontractive semigroups under compact
perturbation, {\it Proc. Amer. Math. Soc.}, 105 (1989), 375-383.

\bibitem{Wu2001} S.T. Wu, Virtual passive control of flexible arms with collocated and
noncollocated feedback, {\it J. Robot. Syst.}, 18 (2001), 645-655.

\bibitem{Zhao2010} X.W. Zhao, G. Weiss, Well-posedness, regularity and exact controllability of the SCOLE model,
 {\it Mathematics of Control, Signals, and Systems}, 22(2) (2010), 91-127.

\bibitem{Zhou2017a} H.C. Zhou, B.Z. Guo, Unknown input observer design and output feedback
stabilization for multi-dimensional wave equation with boundary control
matched uncertainty, {\it Journal of Differential Equations}, 263 (2017), 2213-2246.

\bibitem{Zhou2018a} H.C. Zhou, H. Feng, Disturbance estimator based output feedback exponential
stabilization for Euler-Bernoulli beam equation with boundary control, {\it Automatica}, 91 (2018), 79-88.

\bibitem{Zhou2018b} H.C. Zhou, G. Weiss, Output feedback exponential stabilization for
one-dimensional unstable wave equations with boundary control matched
disturbance, {\it SIAM J. Control Optim.}, 56 (2018), 4098-4129.

\bibitem{Zhou2020} H.C. Zhou and H. Feng, Stabilization for Euler-Bernoulli beam equation with boundary moment control and disturbance
via a new disturbance estimator, {\it J. Dyn. Control. Syst.}, 27 (2021), 247-259.


\end{thebibliography}
\end{document}